\documentclass[11pt]{article}

\usepackage{amsmath,amssymb,amsthm,amscd, bbm, comment, enumitem, stmaryrd}
\usepackage{latexsym}
\usepackage[all]{xy}

\usepackage{tikz}

\allowdisplaybreaks[1]

 \theoremstyle{plain}
\newtheorem{thm}{Theorem}[section]
\newtheorem{lemma}[thm]{Lemma}
\newtheorem{proposition}[thm]{Proposition}
\newtheorem{cor}[thm]{Corollary}

\theoremstyle{definition}
\newtheorem*{defn}{Definition}

\newtheorem{example}{Example}

\xyoption{dvips}
\numberwithin{equation}{section}
\newcommand{\End}{\mathrm{End}}
\newcommand{\Hom}{\mathrm{Hom}}

\newcommand{\Res}{\mathrm{Res}}

\newcommand{\SInd}{\mathrm{SInd}}
\newcommand{\Ind}{\mathrm{Ind}}

\newcommand{\nst}{\mathrm{nst}}
\newcommand{\Nst}{\mathrm{Nst}}
\newcommand{\crs}{\mathrm{crs}}
\newcommand{\Crs}{\mathrm{Crs}}
\newcommand{\xfrown}{{}_{\times}\hspace{-1mm}\frown}
\newcommand{\sh}{\mathrm{sh}}
\newcommand{\frownx}{\frown_{\hspace{-1mm}\times}}

\DeclareMathOperator{\Cspan}{\mathbb{C}-span}
\newcommand{\tr}{\mathrm{tr}}

\def\Arc{\begin{tikzpicture}
 \draw[line width=.15mm, color=white] (.15,0) to[in=60, out=120, circle, looseness=1.8] (-.45,0);
 \draw[fill=black] (-.45,0) circle (.045);
\draw[fill=black] (-.15,0) circle (.045);
\draw[fill=black] (.15,0) circle (.045);
\end{tikzpicture}}

 \def\arc{\begin{tikzpicture}
  \draw[line width=.15mm, color=white] (-.15,0) to[in=60, out=120, circle, looseness=1.8] (-.45,0);
 \draw[fill=black] (-.45,0) circle (.045);
\draw[fill=black] (-.15,0) circle (.045);
\end{tikzpicture}}

\def\aRc{\begin{tikzpicture}
 \draw[line width=.15mm, color=white] (.15,0) to[in=60, out=120, circle, looseness=1.8] (-.15,0);
\draw[fill=black] (-.15,0) circle (.045);
\draw[fill=black] (.15,0) circle (.045);
\draw[line width=.15mm] (.15,0) to[in=60, out=120, circle, looseness=1.8] (-.15,0);
\end{tikzpicture}}

\def\ARc{\begin{tikzpicture}
 \draw[line width=.15mm, color=white] (.15,0) to[in=60, out=120, circle, looseness=1.8] (-.45,0);
 \draw[fill=black] (-.15,0) circle (.045);
\draw[fill=black] (.15,0) circle (.045);
\draw[fill=black] (-.45,0) circle (.045);
\draw[line width=.15mm] (-.15,0) to[in=60, out=120, circle, looseness=1.8] (-.45,0);
\draw[line width=.15mm] (.15,0) to[in=60, out=120, circle, looseness=1.8] (-.15,0);
\end{tikzpicture}}

\def\ArC{\begin{tikzpicture}
 \draw[line width=.15mm, color=white] (.15,0) to[in=60, out=120, circle, looseness=1.8] (-.45,0);
\draw[fill=black] (-.45,0) circle (.045);
\draw[fill=black] (-.15,0) circle (.045);
\draw[fill=black] (.15,0) circle (.045);
\draw[line width=.15mm] (.15,0) to[in=60, out=120, circle, looseness=1.8] (-.15,0);
\end{tikzpicture}}

\def\aRC{\begin{tikzpicture}
 \draw[line width=.15mm, color=white] (.15,0) to[in=60, out=120, circle, looseness=1.8] (-.45,0);
 \draw[fill=black] (-.45,0) circle (.045);
\draw[fill=black] (-.15,0) circle (.045);
\draw[fill=black] (.15,0) circle (.045);
\draw[line width=.15mm] (.15,0) to[in=55, out=125, circle, looseness=1.5] (-.45,0);
\end{tikzpicture}}

\def\ARC{\begin{tikzpicture}
 \draw[line width=.15mm, color=white] (.15,0) to[in=60, out=120, circle, looseness=1.8] (-.45,0);
 \draw[fill=black] (-.45,0) circle (.045);
\draw[fill=black] (-.15,0) circle (.045);
\draw[fill=black] (.15,0) circle (.045);
\draw[line width=.15mm] (-.15,0) to[in=60, out=120, circle, looseness=1.8] (-.45,0);
\end{tikzpicture}}

\DeclareMathOperator{\wt}{wt}

\begin{document}

\title{Shell Tableaux:\\ A set partition analog of vacillating tableaux}
\author{Megan Ly\\ 
 }

\date{}

\maketitle

\begin{abstract}
Schur--Weyl duality is a fundamental framework in combinatorial representation theory.  It intimately relates the irreducible representations of a group to the irreducible representations of its centralizer algebra.  We investigate the analog of Schur--Weyl duality for the group of unipotent upper triangular matrices over a finite field.  In this case, the character theory of these upper triangular matrices is ``wild'' or unattainable.  Thus we employ a generalization, known as supercharacter theory, that creates a striking variation on the character theory of the symmetric group with combinatorics built from set partitions.  In this paper, we present a combinatorial formula for calculating a restriction and induction of supercharacters based on statistics of set partitions and seashell inspired diagrams.  We use these formulas to create a graph that encodes the decomposition of a tensor space, and develop an analog of Young tableaux, known as shell tableaux, to index paths in this graph. 
 \end{abstract}

\section{Introduction}


Schur--Weyl duality forms an archetypal situation in combinatorial representation theory involving two actions that complement each other.  In the basic setup, a $G$-module $M$ of a finite group $G$ is tensored together $k$ times to form the tensor space
 \[M^{\otimes k} = \underbrace{M\otimes\cdots\otimes M}_{k{\rm \  factors}}.\]
 The commuting actions of $G$ and its centralizer algebra $Z_k = \End_G(M^{\otimes k})$ on $M^{\otimes k}$ produce a decomposition
\[ M^{\otimes k} \cong \bigoplus_{\lambda} G^\lambda \otimes Z^\lambda_k \quad \text{as a} \ (G, Z_k)\text{-bimodule} \]
where the $G^\lambda$ are irreducible $G$-modules and the $Z^\lambda_k$ are irreducible $Z_k$-modules.  This bimodule decomposition intimately relates the irreducible representations of $G$ with the irreducible representations of $Z_k$.

In the classical situation, the general linear group $GL_n(\mathbb{C})$ of $n \times n$ matrices over the field $\mathbb{C}$ of complex numbers acts on the tensor space $V^{\otimes k}$ of an $n$ dimensional vector space $V$, and its centralizer algebra is the symmetric group $S_k$ on the $k$ tensor factors.  More recently, the study of new versions of Schur--Weyl duality has led to many remarkable discoveries about algebras of operators on tensor space that are full centralizers of each other.  For example,
\begin{enumerate}
\item the Brauer algebra is the centralizer of the symplectic and orthogonal groups acting on the tensor space $({\mathbb{C}^n})^{\otimes k}$ \cite{Brauer};

\item the Temperley-Lieb algebra is the centralizer of the special linear Lie group of degree two acting on the tensor space $(\mathbb{C}^2)^{\otimes k}$ \cite{temperley};

\item the partition algebra is the centralizer of the symmetric group acting on the tensor space $V^{\otimes k}$ of its permutation representation $V$ \cite{Partition}.
\end{enumerate}

\noindent This paper focuses on a unipotent analog of Schur--Weyl duality.  

For a positive integer $n$ and a power of a prime $q= p^r$, consider the finite group of unipotent $n \times n$ upper-triangular matrices
\[ U_n  = \left\{\begin{bmatrix}
1 & * & \cdots & * \\
0 & 1 &  & \vdots\\
\vdots & & \ddots & *\\
0 & \cdots & 0 & 1
\end{bmatrix} \right\}\]  
with ones on the diagonal and entries $\ast$ in the finite field $\mathbb{F}_q$ with $q$ elements.  Since $U_n$ is a Sylow $p$-subgroup of $GL_n(\mathbb{F}_q)$, then every $p$-group of $GL_n(\mathbb{F}_q)$ is conjugate to a subgroup of $U_n$.  Embedding every finite $p$-group in $S_n \subseteq GL_n(\mathbb{F}_q)$ as permutation matrices, it follows that every $p$-group is isomorphic to a subgroup of $U_n$.  This is akin to how every finite group is isomorphic to a subgroup of $S_n$, so it is not unreasonable to hope that the representation theories of $U_n$ and $S_n$ have comparable structures.

Unlike the combinatorially rich representation theory of $S_n$ \cite{Mac}, the representation theory of $U_n$ is well-known to be intractable or ``wild'' \cite{wild}. Nevertheless, Andr\'e \cite{Andre1, Andre2, Andre3, Andre4} and Yan \cite{Yan} constructed a workable approximation that has been useful in studying Fourier analysis \cite{supercharacter}, random walks \cite{ADS}, and Hopf algebras \cite{Hopf}.  In \cite{supercharacter} Diaconis and Isaacs generalize this idea to arbitrary finite groups to develop the notion of supercharacter theory.  Supercharacter theory approximates the character theory of a finite group by replacing conjugacy classes with certain unions of conjugacy classes called ``superclasses'' and irreducible characters with certain linear combinations of irreducible characters called ``supercharacters''.

We study a coarsening of Andre and Yan's traditional super-representation theory on $U_n$ \cite{rainbow} where there is a one-to-one correspondence between
 \[\left\{ \begin{array}{c} \text{supercharacters}\\ \text{ of $U_n$}\end{array}\right\} \longleftrightarrow \left\{ \begin{array}{c} \text{Set partitions of}\\  \, \{1,2, \ldots, n\} \end{array}\right\}. \] 
It is becoming ever more apparent that the set partition combinatorics of this super-representation theory is analogous to the classical partition combinatorics of the representation theory of the symmetric group, but with some important differences.

We first study the decomposition of $V^{\otimes k}$ where $V= \mathbb{C}U_n \otimes_{\mathbb{C}U_{n-1}} \mathbbm{1}$ as a $U_n$-supermodule.  Much like the partition algebra, we have
\[V^{\otimes k} \cong \underbrace{(\Ind_{U_{n-1}}^{U_n} \Res_{U_{n-1}}^{U_n})}_{k{\rm \  times}}\hspace{-1mm}{}^k(\mathbbm{1}) \]
where the trivial supercharacter is restricted and induced $k$ times.  We provide a combinatorial formula calculating a restriction of supercharacters from $U_n$ to $U_{n-1}$ where the coefficients of the supercharacters of $U_{n-1}$ are a product of powers of $q$ and $q-1$ based on statistics of set partitions and seashell inspired diagrams.  For example, a shell formed by two set partitions is shown below
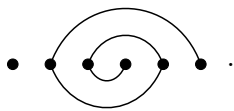
\[\begin{tikzpicture}
 \draw[fill=black] (-1,0) circle (.07);
 \draw[fill=black] (-.5,0) circle (.07);
\draw[fill=black] (.0,0) circle (.07);
\draw[fill=black] (.5,0) circle (.07);
\draw[fill=black] (1,0) circle (.07);
\draw[fill=black] (1.5,0) circle (.07);
\draw[line width=.2mm] (0,0) to[in=-110, out=-70, circle, looseness=1.6] (.5,0);
\draw[line width=.2mm] (0,0) to[in=110, out=70, circle, looseness=1.4] (1,0);
\draw[line width=.2mm] (-.5,0) to[in=110, out=70, circle, looseness=1.35] (1.5,0);
\draw[line width=.2mm] (-.5,0) to[in=-110, out=-70, circle, looseness=1.4] (1,0);
\node at (1.9,0) {.};
\end{tikzpicture}\]
Using Frobenius reciprocity, we obtain a corresponding formula for inducing supercharacters.  Together these formulas are known as branching rules.  As opposed to the representation theory of the symmetric group, they depend on the embedding of $U_{n-1}$ in $U_n$.

We then use the branching rules to create a graph that encodes the decomposition of $V^{\otimes k}$ known as the Bratteli diagram.  Since we are approximating by supercharacters, the Bratteli diagram produces a decomposition of a subalgebra of the centralizer algebra that treats supermodules as irreducibles.  For the partition algebra, paths in the Bratteli diagram are indexed by a set of combinatorial objects called vacillating tableaux.  We create an analog of vacillating tableaux, known as shell tableaux, built from a generalization of shells.  Next, we construct a bijection between shell tableaux and paths in the Bratteli diagram.  When $q=2$, we remove a condition on shell tableaux to produce a bijection with weighted paths in the Bratteli diagram.  In contrast with the symmetric group, these weights account for the multiplicities in our Bratteli diagram.  On the whole, the shell combinatorics developed from this paper may help compute in other algebraic structures related to the supercharacter theory of $U_n$, such as the Hopf algebra of symmetric functions in noncommuting variables.

\vspace{.25cm}
 
 \noindent\textbf{Acknowledgements.}  I would like to thank Nat Thiem for all his help and guidance. 

%
%

\section{Preliminaries}\label{Preliminaries}

This section reviews a supercharacter theory for the group of unipotent upper triangular matrices and the combinatorics of its representation theory based on set partitions.

\subsection{A supercharacter theory for $U_n$}
A \textit{supercharacter theory} of a group $G$ consists of a set of \textit{superclasses} $\mathcal{K}$ and a set of \textit{supercharacters} $\mathcal{X}$ such that
\begin{enumerate}[label=(\alph*)]
\item the set $\mathcal{K}$ is a partition of $G$ into unions of conjugacy classes,
\item the set $\mathcal{X}$ is a set of characters such that each irreducible character of $G$ is a constituent of exactly one supercharacter,
\item $|\mathcal{K}|= |\mathcal{X}|$,
\item the supercharacters are constant on superclasses.
\end{enumerate}

Every group $G$ has two ``trivial'' supercharacter theories: the usual character theory, and the supercharacter theory with $\mathcal{K} = \{ \{1\}, G-\{1\} \} $ and $\mathcal{X} = \{\mathbbm{1}, \chi_{\text{reg}} -\mathbbm{1}\}$ where $\mathbbm{1}$ is the trivial character of $G$ and $\chi_{\text{reg}}$ is the regular character.  While many finite groups have several supercharacter theories \cite{supercharacter}, preference is given to supercharacter theories that strike a balance between computability and producing better approximations of the usual character theory.

We focus on the supercharacter theory on $U_n$ given in \cite{branching rules} that is a slight coarsening of the traditional supercharacter theory of Andr\'e and Yan.

Let $U_n$ be the subgroup of unipotent upper-triangular matrices of the general linear group $GL_n(\mathbb{F}_q)$ over the finite field $\mathbb{F}_q$ with $q$ elements, $B_n$ be the normalizer of $U_n$ in $GL_n(\mathbb{F}_q)$ consisting of upper triangular matrices, and
\[\mathfrak{u}_n = U_n - 1\]
be the nilpotent $\mathbb{F}_q$-algebra of strictly upper triangular matrices.  The subgroup $B_n$ acts by left and right multiplication on $\mathfrak{u}_n$, and the superclasses are given by the two-sided orbits
\[\begin{array}{ccc}
B_n \mathfrak{u}_n B_n & \longleftrightarrow& \mathcal{K} \\
B_n x B_n &\mapsto&1+ B_n x  B_n.
 \end{array}\]

Following the construction in \cite{rainbow}, fix a nontrivial homomorphism $\vartheta: \mathbb{F}_q^+ \to \mathbb{C}^\times$.  The $\mathbb{F}_q$-vector space of $n \times n$ matrices $\mathfrak{gl}_n(\mathbb{F}_q)$ decomposes in terms of upper triangular matrices $\mathfrak{b}_n$ and strictly lower triangular matrices $\mathfrak{l}_n$ as
\[ \mathfrak{gl}_n = \mathfrak{b}_n \oplus \mathfrak{l}_n .\]
Identifying $\mathfrak{l}_n$ with $\mathfrak{gl}_n/\mathfrak{b}_n$ makes $\mathfrak{l}_n$ a canonical set of coset representatives in $\mathfrak{gl}_n/\mathfrak{b}_n$.  For $v \in \mathfrak{gl}_n$ define
\[ \bar{v} = (v+\mathfrak{b}_n) \cap \mathfrak{l}_n .\]
Then for $v \in \mathfrak{l}_n$, 
\[ \Cspan \{\overline{a v} \mid a \in B_n  \} \]
is $U_n$-supermodule with left action 
\[ u w = \vartheta(\tr((u-1)w) (\overline{uw})\quad \text{for} \quad u \in U_n, w \in \mathfrak{l}_n \]
and right action
\[ w u = \vartheta(\tr(w(u^{-1}-1)) (\overline{wu^{-1}})\quad \text{for} \quad u \in U_n, w \in \mathfrak{l}_n. \]
The two-sided orbits from extending these actions on $\mathfrak{l}_n$ to the normalizer subgroup $B_n$ yields corresponding supercharacters given by,
\[\begin{array}{ccc}
B_n \mathfrak{l}_n B_n & \longleftrightarrow& \mathcal{X} \\
B_n v B_n &\mapsto& \quad g\mapsto \displaystyle{\frac{|B_n v|}{|B_n v B_n|}} \sum_{ w \in B_n v B_n} \vartheta(\tr((g -1) w)).
 \end{array}\]

In constructing the supercharacters of $U_n$ it is more common to construct a module structure on the dual $\mathfrak{u}_n^*$, where $\mathfrak{u}_n = U_n -1$ as in \cite{supercharacter}.  However, the actions of $B_n$ on $\mathfrak{l}_n$ are a translation of the actions on $\mathfrak{u}^*_n$ that make studying modules more straightforward \cite{rainbow}.

By elementary row and column operations we may choose orbit representatives for the two-sided action of $B_n$ on $\mathfrak{u}_n$ and $\mathfrak{l}_n$ so that there is a one to one correspondence between
 \[\left\{ \begin{array}{c} \text{superclasses}\\ \text{ of $U_n$}\end{array}\right\} \longleftrightarrow \left\{u \in U_n\ \Bigg|  \begin{array}{c} u-1 \ \text{has at most one 1}\\ \text{in every row and column} \end{array}\right\} \] 
 \[\left\{ \begin{array}{c} \text{supercharacters}\\ \text{ of $U_n$}\end{array}\right\} \longleftrightarrow \left\{v \in \mathfrak{l}_n\ \Bigg|  \begin{array}{c} v \ \text{has at most one 1}\\ \text{in every row and column} \end{array}\right\}. \] 
These representatives are indexed by set partitions.

If instead of considering the orbits of the full subgroup $B_n$, we consider the $U_n$ orbits on the group $\mathfrak{u}_n$ and its dual $\mathfrak{u}_n^\ast$, then we obtain the traditional supercharacter theory of Andr\'e and Yan.  In this case the combinatorics depends on the finite field $\mathbb{F}_q$ and is based on $\mathbb{F}_q^\times$-colored set partitions.

\subsection{Set Partition Combinatorics}\label{set partition}

Define $[n] = \{ 1,2, \ldots, n\}$.  A \textit{set partition} $\lambda$ of $[n]$ is a subset $\{(i,j) \in [n] \times [n] \mid i< j\}$ such that if $(i,k),(j,l) \in \lambda$, then $i=j$ if and only if $k=l$.  We represent each set partition $\lambda\vdash [n]$ diagrammatically as a set of arcs on a row of $n$ nodes so that if $(i,j) \in \lambda$, then there is an arc connecting the $i$th node to the $j$th node.  For example,
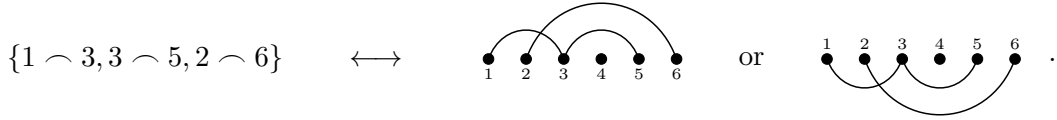
\[\begin{tikzpicture}
\node at (-4,0) {$\{1\frown 3 , 3 \frown 5 , 2 \frown 6 \}
\qquad\longleftrightarrow\qquad 
\qquad$};
\draw[fill=black] (-1,0) circle (.07);
\node at (-1,-.2) {\tiny 1};
 \draw[fill=black] (-.5,0) circle (.07);
 \node at (-.5,-.2) {\tiny 2};
\draw[fill=black] (.0,0) circle (.07);
\node at (0,-.2) {\tiny 3};
\draw[fill=black] (.5,0) circle (.07);
\node at (.5,-.2) {\tiny 4};
\draw[fill=black] (1,0) circle (.07);
\node at (1,-.2) {\tiny 5};
\draw[fill=black] (1.5,0) circle (.07);
\node at (1.5,-.2) {\tiny 6};
\draw[line width=.2mm] (-1,0) to[in=110, out=70, circle, looseness=1.4] (0,0);
\draw[line width=.2mm] (-.5,0) to[in=110, out=70, circle, looseness=1.35] (1.5,0);
\draw[line width=.2mm] (0,0) to[in=110, out=70, circle, looseness=1.4] (1,0);
\node at (2.5,0) {or};
\draw[fill=black] (3.5,0) circle (.07);
\node at (3.5,.2) {\tiny 1};
 \draw[fill=black] (4,0) circle (.07);
 \node at (4, .2) {\tiny 2};
\draw[fill=black] (4.5,0) circle (.07);
\node at (4.5,.2) {\tiny 3};
\draw[fill=black] (5,0) circle (.07);
\node at (5,.2) {\tiny 4};
\draw[fill=black] (5.5,0) circle (.07);
\node at (5.5,.2) {\tiny 5};
\draw[fill=black] (6,0) circle (.07);
\node at (6,.2) {\tiny 6};
\draw[line width=.2mm] (3.5,0) to[in=-110, out=-70, circle, looseness=1.4] (4.5,0);
\draw[line width=.2mm] (4,0) to[in=-110, out=-70, circle, looseness=1.35] (6,0);
\draw[line width=.2mm] (4.5,0) to[in=-110, out=-70, circle, looseness=1.4] (5.5,0);
\node at (6.5,0) {.};
\end{tikzpicture}\]
In these diagrams it is natural to draw the arcs above or below the nodes.  We will use both orientations to compare set partitions.  We typically refer to the pair $(i,j)$ as an \textit{arc} in $\lambda$ and write $(i,j)= i\frown j$ or $(i,j) = i \smile j$ to specify the arc.  For each arc $(i, j) \in \lambda$ we call $i$ the \textit{left endpoint} and $j$ the \textit{right endpoint}.  The sets of left and right endpoints of $\lambda$ are given by
\[ le(\lambda) = \{i \in [n] \mid (i, j) \in \lambda, \ \text{for some}\ j \in [n]\}\]\[re(\lambda) = \{j \in [n] \mid (i, j) \in \lambda, \ \text{for some}\ i \in [n]\}. \]
 We say two arcs \textit{conflict} if they have the same the same left or right endpoints.  Thus no arcs conflict in a set partition.
 
 We obtain the more traditional definition of set partitions by taking $\textrm{part}(\lambda)$ for $\lambda \vdash [n]$ to be the set of equivalence classes on $[n]$ given by the reflexive transitive closure of $i \sim j$ if $(i, j) \in \lambda$.  For instance,
 \[\begin{tikzpicture}
\node at (-1.6,0) {part\Bigg(};
 \draw[fill=black] (-1,0) circle (.07);
\node at (-1,-.2) {\tiny 1};
 \draw[fill=black] (-.5,0) circle (.07);
 \node at (-.5,-.2) {\tiny 2};
\draw[fill=black] (.0,0) circle (.07);
\node at (0,-.2) {\tiny 3};
\draw[fill=black] (.5,0) circle (.07);
\node at (.5,-.2) {\tiny 4};
\draw[fill=black] (1,0) circle (.07);
\node at (1,-.2) {\tiny 5};
\draw[fill=black] (1.5,0) circle (.07);
\node at (1.5,-.2) {\tiny 6};
\draw[line width=.2mm] (-1,0) to[in=110, out=70, circle, looseness=1.4] (0,0);
\draw[line width=.2mm] (-.5,0) to[in=110, out=70, circle, looseness=1.35] (1.5,0);
\draw[line width=.2mm] (0,0) to[in=110, out=70, circle, looseness=1.4] (1,0);
\node at (3.8,0) {\Bigg) = $\{\{1,3,5\}, \{ 2,6\}, \{4\} \}$.};
 \end{tikzpicture} \]
Note the connected components of the diagram are the parts of the set partition and the arcs are the adjacent pairs of elements in each part.

There are some natural statistics on set partitions \cite{diaconis}.  For a set partition $\lambda \vdash [n]$ the \textit{dimension} is 
$$\dim (\lambda) = \sum_{i \frown j\in \lambda} j-i -1.$$
For a pair of set partitions $\lambda, \mu \vdash [n]$ define
\begin{eqnarray*}
&\Crs(\lambda, \mu) = \{ ((i, k) , (j , l))\in \lambda \times \mu \mid i<j<k<l\} , \qquad& \crs(\lambda, \mu) = |\Crs(\lambda, \mu)|,\\
&\Nst_\mu^\lambda = \{ ((i , l) , (j , k))\in \lambda \times \mu \mid i<j<k<l\} ,\qquad \qquad& \nst_\mu^\lambda = |\Nst_\mu^\lambda |  
\end{eqnarray*}
as the \textit{crossing set}, \textit{crossing number}, \textit{nesting set}, and \textit{nesting number} respectively.  To illustrate, if
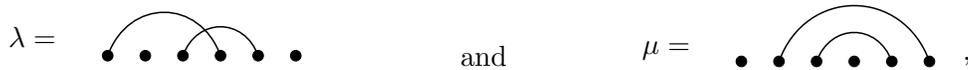
\[\begin{tikzpicture}
\node at (-2,.25) {$\lambda =$};
 \draw[fill=black] (-1,0) circle (.07);
 \draw[fill=black] (-.5,0) circle (.07);
\draw[fill=black] (.0,0) circle (.07);
\draw[fill=black] (.5,0) circle (.07);
\draw[fill=black] (1,0) circle (.07);
\draw[fill=black] (1.5,0) circle (.07);
\draw[line width=.2mm] (-1,0) to[in=110, out=70, circle, looseness=1.4] (.5,0);
\draw[line width=.2mm] (0,0) to[in=110, out=70, circle, looseness=1.4] (1,0);
\node at (4,0) {and};
 \end{tikzpicture}\qquad \qquad  \begin{tikzpicture}
\node at (-2,.15) {$\mu =$};
 \draw[fill=black] (-1,0) circle (.07);
 \draw[fill=black] (-.5,0) circle (.07);
\draw[fill=black] (.0,0) circle (.07);
\draw[fill=black] (.5,0) circle (.07);
\draw[fill=black] (1,0) circle (.07);
\draw[fill=black] (1.5,0) circle (.07);
\draw[line width=.2mm] (-.5,0) to[in=110, out=70, circle, looseness=1.35] (1.5,0);
\draw[line width=.2mm] (0,0) to[in=110, out=70, circle, looseness=1.4] (1,0);
\node at (2,0) {,};
 \end{tikzpicture}  \]
 then we have
 \[ \dim(\lambda) = 3, \quad \crs(\lambda, \lambda) = 1, \quad \nst_\lambda^\lambda = 0, \qquad \qquad \dim(\mu) = 4, \quad \crs(\mu, \mu) =0, \quad \nst_\mu^\mu = 1.\]
Superimposing $\lambda$ and $\mu$, where the arcs of $\lambda$ are dashed
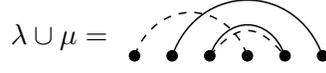
\[\begin{tikzpicture}
\node at (-2,.25) {$\lambda \cup \mu =$};
 \draw[fill=black] (-1,0) circle (.07);
 \draw[fill=black] (-.5,0) circle (.07);
\draw[fill=black] (.0,0) circle (.07);
\draw[fill=black] (.5,0) circle (.07);
\draw[fill=black] (1,0) circle (.07);
\draw[fill=black] (1.5,0) circle (.07);
\draw[line width=.2mm, style=dashed] (-1,0) to[in=110, out=70, circle, looseness=1.4] (.5,0);
\draw[line width=.2mm] (-.5,0) to[in=110, out=70, circle, looseness=1.35] (1.5,0);
\draw[line width=.2mm] (0,0) to[in=110, out=70, circle, looseness=1.5] (1,0);
\draw[line width=.2mm, style=dashed] (0,0) to[in=130, out=50, circle, looseness=1.5] (1,0);
\end{tikzpicture} \]
yields
\[\Crs(\lambda, \mu) = \{(1\frown 4,2 \frown 6), (1\frown 4, 3\frown 5)\}, \qquad \Nst^\lambda_\mu = \emptyset\]
but,
\[\Crs(\mu, \lambda) = \emptyset, \qquad \Nst^\mu_\lambda = \{(2\frown 6, 3\frown 5)\}.\]
While it is not generally true that $\Crs(\lambda, \mu) = \Crs(\mu, \lambda)$, it follows from the definition of a crossing number that for all set partitions $\lambda, \mu, \nu \vdash [n]$,
\begin{eqnarray}\label{crossingidentity}
\Crs(\lambda, \mu \cup \nu) &=& \Crs(\lambda, \mu) + \Crs(\lambda, \nu) \\
\Crs(\lambda \cup \mu, \nu) &=& \Crs(\lambda, \nu) + \Crs(\mu, \nu) .
\end{eqnarray}

\subsection{An uncolored supercharacter theory}\label{uncolored}

We describe the correspondence between set partitions and the superclasses and supercharacters of $U_n$.  Given a set partition $\lambda \vdash [n]$, we construct a representative $u_\lambda$ of a superclass of $U_n$ by
\[ (u_\lambda)_{i,j} = \left\{ \begin{array}{cl} 1 & \text{if} \ i\frown j \in \lambda\ \text{or}\ i = j, \\ 0 & \text{otherwise.}\end{array}\right. \]
For instance, the correspondence between $\lambda$ and $u_\lambda$ is given as follows 
\[\begin{tikzpicture}
\node at (-2,.25) {$\lambda =$};
 \draw[fill=black] (-1,0) circle (.07);
 \draw[fill=black] (-.5,0) circle (.07);
\draw[fill=black] (.0,0) circle (.07);
\draw[fill=black] (.5,0) circle (.07);
\draw[fill=black] (1,0) circle (.07);
\draw[fill=black] (1.5,0) circle (.07);
\draw[line width=.2mm] (-1,0) to[in=110, out=70, circle, looseness=1.4] (0,0);
\draw[line width=.2mm] (-.5,0) to[in=110, out=70, circle, looseness=1.35] (1.5,0);
\draw[line width=.2mm] (0,0) to[in=110, out=70, circle, looseness=1.4] (1,0);
\end{tikzpicture} \qquad \longleftrightarrow \qquad u_\lambda-1 = \left[\begin{array}{cccccc} 0 & 0 & 1 & 0 & 0& 0\\ 0 & 0 & 0 & 0 &0 & 1\\ 0 & 0 & 0 & 0 & 1 & 0\\ 0 & 0 & 0 & 0 & 0 & 0 \\ 0& 0 & 0 & 0 & 0 & 0\\0& 0 & 0 & 0 & 0 & 0\\\end{array}\right].\]
The corresponding superclass $\mathcal{K}_\lambda$ is
\[\mathcal{K}_\lambda = 1+ B_n (u_\lambda-1) B_n. \]

Similarly, a representative $v_\lambda$ for the two-sided action of $B_n$ on $\mathfrak{l}_n$ is
\[ (v_\lambda)_{k,j} = \left\{ \begin{array}{cl} 1 & \text{if} \ j\frown k \in \lambda, \\ 0 & \text{otherwise}\end{array}\right. \]
so that
\[V^\lambda \cong \Cspan \{ \overline{av_\lambda} \mid a \in B_n \}\]
and for $g \in U_n$, the corresponding supercharacter $\chi^\lambda$ is defined as
\[ \chi^{\lambda}(g) = \frac{|B_n v_\lambda|}{|B_n v_\lambda B_n|} \sum_{v \in B_n v_\lambda B_n} \vartheta(\tr((g-1) \ v)).\] 

Amazingly, many properties of these supercharacters can be determined using statistics of set partitions.

\begin{proposition}[{\cite[Bragg, Thiem, Proposition 2.1]{rainbow}}]\label{uncoloredsupercharacter}
For $\lambda, \mu \vdash [n]$, we have
\begin{equation*}
\chi^{\lambda}(u_\mu)=\left\{\begin{array}{ll}\displaystyle \frac{(-1)^{|\lambda \cap \mu|}q^{\dim(\lambda)}(q-1)^{|\lambda - \mu|}}{q^{\nst_\mu^\lambda}}  & \begin{array}{@{}l}\text{\text{if}\ $i<j<k$, $i\frown k\in \lambda$}\\ \text{\text{then}\ $i\frown j,j\frown k\notin \mu$,}\end{array}\\ 0 & \text{otherwise.}\end{array}\right.
\end{equation*}
\end{proposition}
In particular the trivial supercharacter $\mathbbm{1}$ is the supercharacter $\chi^\varnothing$ corresponding to the empty set partition of $[n]$, and the degree of each supercharacter is
\[ \chi^\lambda(1) = q^{\dim(\lambda)}(q-1)^{|\lambda| } .\]
It also follows from the formula that supercharacters factor as tensor products of arcs
\begin{equation}\label{arctensor}
\chi^\lambda = \bigodot_{i \frown j \in \lambda} \chi^{i \frown j} \qquad \text{where}\ (\chi \odot \psi)(g) = \chi(g) \psi(g).
\end{equation}

With respect to the inner product the supercharacters form an orthogonal set.
\begin{proposition}\label{superorthogonality}
For $\lambda, \mu \vdash [n]$, we have
\[\langle \chi^\lambda, \chi^\mu \rangle = \delta_{\lambda\mu} (q-1)^{|\lambda|} q^{\crs(\lambda, \lambda)} .\]
\end{proposition}\label{orthogonality}
\noindent Proposition \ref{superorthogonality} can be proved from \cite[Thiem, (2.3)]{branching rules}.
The crossing number $\crs(\lambda, \lambda)$ helps measure how close a supercharacter is to being irreducible.

%
%

\section{Branching Rules}\label{Branching Rules}

An important property of the supercharacters of $U_n$ is that their restriction to any subgroup is a linear combination of supercharacters with nonnegative integer coefficients \cite{supercharacter}.  However, the coefficients in the restriction decompositions are not well understood \cite{branching rules}.  We provide a combinatorial formula for calculating the restriction of supercharacters of $U_n$ to $U_{n-1}$.  Using Frobenius reciprocity, we obtain a corresponding formula for inducing supercharacters.  Since these formulas depend on the number of nonzero elements in the field $\mathbb{F}_q$, fix 
\[t= q-1.\]

\subsection{Restriction}

We consider the restriction of supercharacters from $U_n$ to $U_{n-1}$ by embedding $U_{n-1} \subseteq U_n$ as
\[ U_{n-1} = \{u \in U_n \mid (u-1)_{ij} \neq 0 \ \text{implies}\ i<j<n \} .\]
Since supercharacters decompose into tensor products of arcs (\ref{arctensor}), for $\lambda \vdash [n]$, we have
\[\chi^\lambda = \bigodot_{i \frown l \in \lambda} \chi^{i \frown l} \quad \text{and} \quad \Res^{U_n}_{U_{n-1}}(\chi^\lambda) = \bigodot_{i \frown l \in \lambda} \Res^{U_n}_{U_{n-1}} ( \chi^{i \frown l}). \]
Consequently we compute restrictions for each $\chi^{i \frown l}$ and use the tensor product to glue together the resulting restrictions. 

The restriction of the supercharacter $\chi^{i \frown l}$ is given using the formulas in \cite{branching rules} for computing restrictions in Andr\'e and Yan's traditional supercharacter theory.

\begin{proposition}\label{restrictionarc}
For $1\leq i < l \leq n$, the restriction $\Res^{U_n}_{U_{n-1}} (\chi^{i \frown l})$ is given by
\[  \Res^{U_n}_{U_{n-1}} (\chi^{i \frown l}) = \left\{
     \begin{array}{lr}
       \chi^{i \frown l} & \text{if}\ l \neq n, \\
      t\bigg(\mathbbm{1}+ \displaystyle\sum_{i<k<l} \chi^{i \frown k}\bigg)  & \text{if}\ l = n.
     \end{array}
   \right.\]

\end{proposition}

\begin{proof}
By the formulas for restriction of colored arcs \cite[Theorem 4.5]{branching rules}, for $l \neq n$, we have
\[
\Res^{U_n}_{U_{n-1}}(\chi^{i \frown l}) = \sum_{a \in \mathbb{F}_q^\times} \Res^{U_n}_{U_{n-1}} (\chi^{i \overset{a}\frown l}) = \sum_{a \in \mathbb{F}_q^\times} \chi^{i \overset{a}\frown l} = \chi^{i \frown l},
\]
and for  $l = n$, we have
\[
\Res^{U_n}_{U_{n-1}}(\chi^{i \frown l}) = \sum_{a \in \mathbb{F}_q^\times} \Res^{U_n}_{U_{n-1}} (\chi^{i \overset{a}\frown l}) = \sum_{a \in \mathbb{F}_q^\times}\bigg( \mathbbm{1} + \displaystyle\sum_{\substack{i<k<l\\ b \in \mathbb{F}_q^\times} }\chi^{i \overset{b}\frown k}\bigg) = t\bigg(\mathbbm{1}+ \displaystyle\sum_{i<k<l} \chi^{i \frown k}\bigg) .
\] 
\end{proof}
\noindent Intuitively, restricting an arc corresponds to removing the last node and reattaching the arc in all possible ways.

We now use the tensor product to glue together the resulting restrictions.  For $1 \leq i < l$, define 
\[\chi^{i\frownx l} = t\bigg(\mathbbm{1}+\sum_{i<k<l} \chi^{i \frown k}\bigg) \quad \text{and} \quad \chi^{i\xfrown l} = t\bigg(\mathbbm{1}+\sum_{i<j<l} \chi^{j \frown l}\bigg).\]
Using the formulas in \cite{branching rules} for the colored supercharacter theory yields the following proposition.

\begin{proposition}\label{tensorarcs}
For $1\leq i < l \leq n$ and $1\leq j < k\leq n$ such that $(i,l) \neq (j,k)$,
\[  \chi^{i \frown l} \odot \chi^{j \frown k} = \left\{
     \begin{array}{ll}
       \chi^{\{i \frown l, j\frown k\}} & \text{if}\ k\neq l, i \neq j, \\
      \chi^{i \frown l} \odot \chi^{j \frownx k}  & \text{if}\ i< j <k = l,\\
      \chi^{i \frown l} \odot \chi^{j \xfrown k}  & \text{if}\ i= j <k < l.
     \end{array}
   \right.\]
\end{proposition}

\begin{proof} Let $1\leq i < l \leq n$ and $1\leq j < k\leq n$ such that $(i,l) \neq (j,k)$.  For $k \neq l$ and $i \neq j$, the tensor product $\chi^{i \frown l} \odot \chi^{j\frown k}$ is given by
\[\chi^{i \frown l} \odot \chi^{j\frown k} =  \sum_{a \in \mathbb{F}_q^\times} \sum_{b \in \mathbb{F}_q^\times} \chi^{i \overset{a}\frown l} \odot \chi^{j \overset{b}\frown k} =  \sum_{a \in \mathbb{F}_q^\times} \sum_{b \in \mathbb{F}_q^\times} \chi^{\{i \overset{a}\frown l, j \overset{b}\frown k\}} = \chi^{\{i \frown l , j \frown k\}},\] 
for $i<j<k=l$, we have
\[
\chi^{i\frown l} \odot \chi^{j \frown l} = \sum_{a \in \mathbb{F}_q^\times} \sum_{b \in \mathbb{F}_q^\times} \chi^{i \overset{a}\frown l} \odot \chi^{j \overset{b}\frown l}
= \sum_{a \in \mathbb{F}_q^\times} \sum_{b \in \mathbb{F}_q^\times} \chi^{i \overset{a}\frown l} \odot \bigg(\mathbbm{1} + \sum_{\substack{j<k<l\\ c \in \mathbb{F}_q^\times} }\chi^{j \overset{c}\frown k}\bigg)
=\chi^{i \frown l} \odot \chi^{j\frownx l},
\]
and for $i= j < k<l$, we obtain
\[\chi^{i\frown l} \odot \chi^{i \frown k} = \sum_{a \in \mathbb{F}_q^\times} \sum_{b \in \mathbb{F}_q^\times} \chi^{i \overset{a}\frown l} \odot \chi^{i \overset{b}\frown k}
= \sum_{a \in \mathbb{F}_q^\times} \sum_{b \in \mathbb{F}_q^\times} \chi^{i \overset{a}\frown l} \odot \bigg(\mathbbm{1} + \sum_{\substack{i<j<k\\ c \in \mathbb{F}_q^\times} }\chi^{j \overset{c}\frown k}\bigg)
=\chi^{i \frown l} \odot \chi^{i\xfrown k}
\] 
by the tensor formulas for colored arcs \cite[Lemma 4.6]{branching rules}.
\end{proof}
\noindent Thus the tensor product provides a rule for resolving conflicting arcs that have the same right endpoint by removing the smaller arc and reattaching it in all possible ways. 

Next we work toward providing a combinatorial description of the coefficients in the tensor product based on  statistics of set partitions and seashell inspired diagrams.
\begin{defn}\label{simpleshell}
Let $s^\prime \in \{ s, s+1\}$ for $s \in \mathbb{Z}_{\geq 1}$ and $1 \leq i \leq l \leq n$.  A \textit{shell} of \textit{size} $n$ and \textit{width} $l-i$ is a set of arcs on $n$ nodes of the form
\[ \bigcup_{r=1}^s \{i_r \frown l_r\} \cup \bigcup_{r=1}^{s^\prime-1} \{i_r \smile l_{r+1}\} \]
where $i = i_1 < \cdots < i_s \leq l_{s^\prime} < \cdots < l_1 = l$.
\end{defn}

\noindent For example, some shells of size $6$ and width $6-2$ are
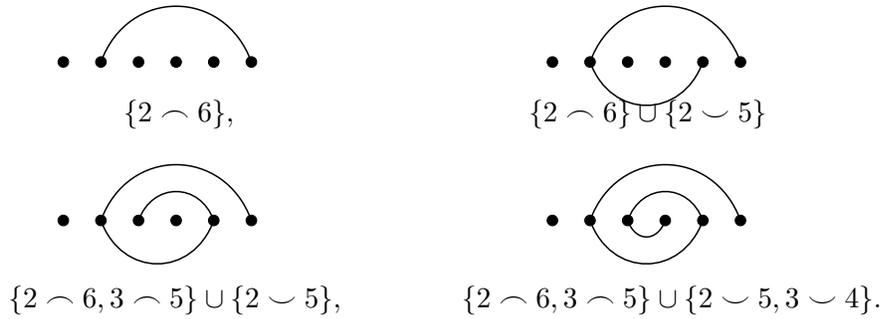
{\[\begin{tikzpicture}
\draw[line width=.2mm, color=white] (-.5,0) to[in=-110, out=-70, circle, looseness=1.4] (1,0);
 \draw[fill=black] (-1,0) circle (.07);
 \draw[fill=black] (-.5,0) circle (.07);
\draw[fill=black] (.0,0) circle (.07);
\draw[fill=black] (.5,0) circle (.07);
\draw[fill=black] (1,0) circle (.07);
\draw[fill=black] (1.5,0) circle (.07);
\draw[line width=.2mm] (-.5,0) to[in=110, out=70, circle, looseness=1.35] (1.5,0);
\end{tikzpicture}\qquad\qquad\qquad\qquad\qquad\begin{tikzpicture}
 \draw[fill=black] (-1,0) circle (.07);
 \draw[fill=black] (-.5,0) circle (.07);
\draw[fill=black] (.0,0) circle (.07);
\draw[fill=black] (.5,0) circle (.07);
\draw[fill=black] (1,0) circle (.07);
\draw[fill=black] (1.5,0) circle (.07);
\draw[line width=.2mm] (-.5,0) to[in=110, out=70, circle, looseness=1.35] (1.5,0);
\draw[line width=.2mm] (-.5,0) to[in=-110, out=-70, circle, looseness=1.4] (1,0);
\end{tikzpicture}\vspace{-.3in}
\]
\[\quad \qquad\{ 2 \frown 6\},\qquad\qquad\qquad\qquad \qquad\{ 2 \frown 6\}\cup\{2 \smile 5\}\]
\[ \begin{tikzpicture}
 \draw[fill=black] (-1,0) circle (.07);
 \draw[fill=black] (-.5,0) circle (.07);
\draw[fill=black] (.0,0) circle (.07);
\draw[fill=black] (.5,0) circle (.07);
\draw[fill=black] (1,0) circle (.07);
\draw[fill=black] (1.5,0) circle (.07);
\draw[line width=.2mm] (0,0) to[in=110, out=70, circle, looseness=1.4] (1,0);
\draw[line width=.2mm] (-.5,0) to[in=110, out=70, circle, looseness=1.35] (1.5,0);
\draw[line width=.2mm] (-.5,0) to[in=-110, out=-70, circle, looseness=1.4] (1,0);
\end{tikzpicture}\qquad\qquad\qquad\qquad\qquad\begin{tikzpicture}
 \draw[fill=black] (-1,0) circle (.07);
 \draw[fill=black] (-.5,0) circle (.07);
\draw[fill=black] (.0,0) circle (.07);
\draw[fill=black] (.5,0) circle (.07);
\draw[fill=black] (1,0) circle (.07);
\draw[fill=black] (1.5,0) circle (.07);
\draw[line width=.2mm] (0,0) to[in=-110, out=-70, circle, looseness=1.6] (.5,0);
\draw[line width=.2mm] (0,0) to[in=110, out=70, circle, looseness=1.4] (1,0);
\draw[line width=.2mm] (-.5,0) to[in=110, out=70, circle, looseness=1.35] (1.5,0);
\draw[line width=.2mm] (-.5,0) to[in=-110, out=-70, circle, looseness=1.4] (1,0);
\end{tikzpicture}\]\vspace{-.3in}
\[\qquad\quad\{2 \frown 6, 3 \frown 5\} \cup \{2\smile 5\}, \qquad\qquad \{2 \frown 6, 3\frown 5\} \cup \{ 2 \smile 5 , 3\smile 4\}.\]}

A \textit{whorl} is pair of consecutive arcs $(i \frown l, i \smile j)$ in a shell corresponding to a $360^\circ$ rotation in the spiral configuration.  Following the notation of Definition \ref{simpleshell}, the number of whorls of a shell is
\[\Big\lceil \frac{s+s^\prime-1}{2} \Big\rceil \]
as each arc is half a whorl.  We use the convention that whorls are counted from the right endpoint $l$ spiraling inward.  For instance, in the shell below we count the two whorls $(1\frown 5, 1\smile 4)$ and $(2\frown 4, 2\smile 3)$ as follows
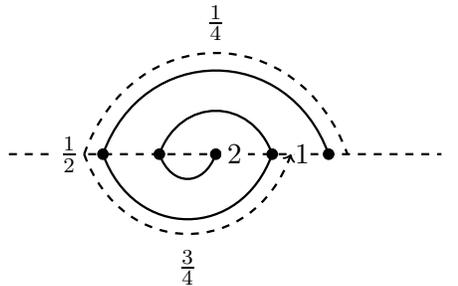
\[\begin{tikzpicture}
\draw[line width=.3mm, style=dashed] (-2,0)--(3.75,0);
\draw[line width=.3mm, style=dashed] (-1,0) to[in=110, out=70, circle, looseness=1.4] (2.5,0);
\draw[->][line width=.3mm, style=dashed] (-1,0) to[in=-110, out=-70, circle, looseness=1.4] (1.75,0);
 \draw[fill=black] (-.75,0) circle (.07);
\draw[fill=black] (0,0) circle (.07);
\draw[fill=black] (.75,0) circle (.07);
\draw[fill=black] (1.5,0) circle (.07);
\draw[fill=black] (2.25,0) circle (.07);
\draw[line width=.3mm] (0,0) to[in=-110, out=-70, circle, looseness=1.6] (.75,0);
\draw[line width=.3mm] (0,0) to[in=110, out=70, circle, looseness=1.4] (1.5,0);
\draw[line width=.3mm] (-.75,0) to[in=110, out=70, circle, looseness=1.35] (2.25,0);
\draw[line width=.3mm] (-.75,0) to[in=-110, out=-70, circle, looseness=1.4] (1.5,0);
 \draw[fill=white, color=white] (1.9,0) circle (.1);
\node at (1.9,0) {$1$};
 \draw[fill=white, color=white] (-1.2,0) circle (.17);
\node at (-1.2,0) {$\frac{1}{2}$};
\node at (.75,1.75) {$\frac{1}{4}$};
\node at (.75/2,-1.5) {$\frac{3}{4}$};
 \draw[fill=white, color=white] (1,0) circle (.17);
\node at (1,0) {$2$};
\node at (4,0) {.};
\end{tikzpicture}\]
If the whorls of a shell are given by $(i_1 \frown l_1, i_1 \smile l_2), \ldots, (i_s \frown l_s, i_s \smile l_{s+1})$ we say the pair $(i_1 \frown l_1, i_1 \smile l_2)$ is the \textit{outer whorl} and the other whorls are \textit{inner whorls}.

We can use shells to determine the partitions that appear in the restriction of a supercharacter.  More precisely, drawing the arcs of a partition $\mu \vdash [n-1]$ below the nodes, and identifying the nodes with the leftmost $n-1$ nodes of a partition $\lambda\vdash [n]$ allows us to characterize the partitions with nonzero coefficients in the restriction of $\lambda$ as the partitions $\mu \vdash [n-1]$ such that the symmetric difference between $\lambda$ and $\mu$ form a shell.

\begin{defn} For $\lambda \vdash [n]$ and $1 \leq i < l \leq n$ with $i \notin le(\lambda)$, the \textit{shell set} $C^{\lambda, i \frown l}$ of $\lambda \cup \{i\frown l\}$ is
\begin{equation*}
C^{\lambda, i \frown l} = \big\{ \mu \vdash [n-1] \mid \big((\lambda \cup \{i \frown l\})-\mu\big) \cup \big(\mu-(\lambda \cup \{ i \frown l\})\big)\ \text{is a shell of width}\ l-i\big\}.
\end{equation*}
\end{defn}
\noindent This corresponds to all the ways to reattach the arc $i \frown l$ and ``straighten'' the resulting diagram by resolving all the conflicting arcs that share the same right endpoint.  

\begin{example}\label{seashells}
Suppose $\lambda = \{1 \frown 4, 3\frown 5\} \vdash [6]$.  Consequently, we have \[\lambda\,\cup\, \{2\frown 6\} =
\begin{tikzpicture}
\draw[fill=black] (-1,0) circle (.07);
 \draw[fill=black] (-.5,0) circle (.07);
\draw[fill=black] (.0,0) circle (.07);
\draw[fill=black] (.5,0) circle (.07);
\draw[fill=black] (1,0) circle (.07);
\draw[fill=black] (1.5,0) circle (.07);
\draw[line width=.2mm] (-1,0) to[in=110, out=70, circle, looseness=1.4] (.5,0);
\draw[line width=.2mm] (-.5,0) to[in=110, out=70, circle, looseness=1.35] (1.5,0);
\draw[line width=.2mm] (0,0) to[in=110, out=70, circle, looseness=1.4] (1,0);
\end{tikzpicture}
\]
and
\[C^{\lambda, 2\frown 6} =  \left\{
\begin{tikzpicture}
\draw[fill=black] (-1,0) circle (.07);
 \draw[fill=black] (-.5,0) circle (.07);
\draw[fill=black] (.0,0) circle (.07);
\draw[fill=black] (.5,0) circle (.07);
\draw[fill=black] (1,0) circle (.07);
\draw[line width=.2mm] (-1,0) to[in=110, out=70, circle, looseness=1.4] (.5,0);
\draw[line width=.2mm] (0,0) to[in=110, out=70, circle, looseness=1.4] (1,0);
\end{tikzpicture} \quad,\quad
\begin{tikzpicture}
\draw[fill=black] (-1,0) circle (.07);
 \draw[fill=black] (-.5,0) circle (.07);
\draw[fill=black] (.0,0) circle (.07);
\draw[fill=black] (.5,0) circle (.07);
\draw[fill=black] (1,0) circle (.07);
\draw[line width=.2mm] (-1,0) to[in=110, out=70, circle, looseness=1.4] (.5,0);
\draw[line width=.2mm] (0,0) to[in=110, out=70, circle, looseness=1.4] (1,0);
\draw[line width=.2mm] (-.5,0) to[in=110, out=70, circle, looseness=1.7] (0,0);
\end{tikzpicture} \quad,\quad
\begin{tikzpicture}
\draw[fill=black] (-1,0) circle (.07);
 \draw[fill=black] (-.5,0) circle (.07);
\draw[fill=black] (.0,0) circle (.07);
\draw[fill=black] (.5,0) circle (.07);
\draw[fill=black] (1,0) circle (.07);
\draw[line width=.2mm] (-1,0) to[in=110, out=70, circle, looseness=1.4] (.5,0);
\draw[line width=.2mm] (-.5,0) to[in=110, out=70, circle, looseness=1.4] (1,0);
\end{tikzpicture}
\right\}.
\]
The seashells created by the symmetric differences between $\lambda\cup \{2 \frown 6\}$ and $\mu \in C^{\lambda, 2 \frown 6}$ are shown as solid lines
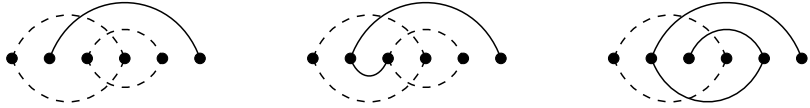
\[\begin{tikzpicture}
\draw[fill=black] (-1,0) circle (.07);
 \draw[fill=black] (-.5,0) circle (.07);
\draw[fill=black] (.0,0) circle (.07);
\draw[fill=black] (.5,0) circle (.07);
\draw[fill=black] (1,0) circle (.07);
\draw[fill=black] (1.5,0) circle (.07);
\draw[line width=.2mm, style=dashed] (-1,0) to[in=110, out=70, circle, looseness=1.4] (.5,0);
\draw[line width=.2mm, style=dashed] (-1,0) to[in=-110, out=-70, circle, looseness=1.4] (.5,0);
\draw[line width=.2mm] (-.5,0) to[in=110, out=70, circle, looseness=1.35] (1.5,0);
\draw[line width=.2mm, style=dashed] (0,0) to[in=110, out=70, circle, looseness=1.4] (1,0);
\draw[line width=.2mm, style=dashed] (0,0) to[in=-110, out=-70, circle, looseness=1.4] (1,0);
\draw[fill=black] (3,0) circle (.07);
 \draw[fill=black] (3.5,0) circle (.07);
\draw[fill=black] (4,0) circle (.07);
\draw[fill=black] (4.5,0) circle (.07);
\draw[fill=black] (5,0) circle (.07);
\draw[fill=black] (5.5,0) circle (.07);
\draw[line width=.2mm, style=dashed] (3,0) to[in=110, out=70, circle, looseness=1.4] (4.5,0);
\draw[line width=.2mm, style=dashed] (3,0) to[in=-110, out=-70, circle, looseness=1.4] (4.5,0);
\draw[line width=.2mm] (3.5,0) to[in=110, out=70, circle, looseness=1.35] (5.5,0);
\draw[line width=.2mm] (3.5,0) to[in=-110, out=-70, circle, looseness=1.7] (4,0);
\draw[line width=.2mm, style=dashed] (4,0) to[in=110, out=70, circle, looseness=1.4] (5,0);
\draw[line width=.2mm, style=dashed] (4,0) to[in=-110, out=-70, circle, looseness=1.4] (5,0);
\draw[fill=black] (7,0) circle (.07);
 \draw[fill=black] (7.5,0) circle (.07);
\draw[fill=black] (8,0) circle (.07);
\draw[fill=black] (8.5,0) circle (.07);
\draw[fill=black] (9,0) circle (.07);
\draw[fill=black] (9.5,0) circle (.07);
\draw[line width=.2mm, style=dashed] (7,0) to[in=110, out=70, circle, looseness=1.4] (8.5,0);
\draw[line width=.2mm, style=dashed] (7,0) to[in=-110, out=-70, circle, looseness=1.4] (8.5,0);
\draw[line width=.2mm] (7.5,0) to[in=110, out=70, circle, looseness=1.35] (9.5,0);
\draw[line width=.2mm] (7.5,0) to[in=-110, out=-70, circle, looseness=1.4] (9,0);
\draw[line width=.2mm] (8,0) to[in=110, out=70, circle, looseness=1.4] (9,0);
\end{tikzpicture}\]
while the arcs in $\lambda \cap \mu$ are dashed.

\end{example}

It will be of interest to examine the shell sets $C^{\lambda, i \frown l}$ by considering the right endpoints $re(\lambda)$.  

\begin{defn}\label{replace} For each $j \frown k \in \lambda$ with $i < j< k < l$ define $\lambda|_{j \mapsto i}$ as the set partition obtained by replacing $j \frown k$ with $i \frown k$ and leaving everything else in $\lambda$ the same.  That is, \[ \lambda|_{j \mapsto i} = \lambda\,\cup\, \{i \frown k\} - \{j \frown k\}.\]
\end{defn}
With this notation we can describe the shell set $C^{\lambda, i \frown l}$ as a union of shells with half a whorl, shells with one whorl, and shells with greater than one whorl.

\begin{lemma}\label{shellset}
For $\lambda \vdash [n]$, and $i \notin le(\lambda)$, the shell set is given by
\[C^{\lambda, i \frown l} = \{ \lambda\} \cup \{\lambda \,\cup\, \{i \frown k\} \mid i<k<l , k \notin re(\lambda)\} \cup \{ \mu \in C^{\lambda\mid_{j \mapsto i}, j \frown k} \mid i<j<k<l, j \frown k \in \lambda\}. \]
\end{lemma}

\begin{proof}
By definition $\{ \lambda,\lambda \cup \{i \frown k\} \mid i<k<l , k \notin re(\lambda)\}\subseteq C^{\lambda, i \frown l}$, so it suffices to show that
\[ C^{\lambda, i \frown l} \backslash  \{ \lambda,\lambda \cup \{i \frown k\} \mid i<k<l , k \notin re(\lambda)\} = \{ \mu \in C^{\lambda\mid_{j \mapsto i}, j \frown k} \mid i<j<k<l, j \frown k \in \lambda\}.\]
 There exist $j = j_1 < \cdots <j_{s}< k_{s^\prime}< \cdots < k_1 = k$ with $s^\prime \in \{s, s+1\}$ such that 
\begin{align*}
(\lambda|_{j \mapsto i} \cup \{j \frown k\})-\mu &= \{ j_1 \frown k_1, j_2 \frown k_2, \ldots,  j_{s} \frown k_s\},\ \text{and}\\
 \mu - (\lambda|_{j \mapsto i} \cup \{j \frown k\}) &= \{ j_1 \frown k_2, j_2 \frown k_3, \ldots, j_{s^\prime -1} \frown k_{s^\prime}\}
 \end{align*}
if and only if there exist $i<j =j_1 < \cdots <j_{s}< k_{s^\prime}< \cdots < k_1 = k<l$ such that 
\begin{align*}
(\lambda \cup \{i \frown l\})-\mu &= \{i \frown l, j_1 \frown k_1, \ldots,  j_{s} \frown k_{s}\},\ \text{and}\\
 \mu - (\lambda \cup \{i \frown l\}) &= \{ i \frown k,  j_1 \frown k_2, \ldots, j_{s^\prime-1} \frown k_{s^\prime}\}.
 \end{align*}
Thus $\mu \in C^{\lambda|_{j \mapsto i}, j \frown k}$ for some $i<j<k<l, j \frown k \in \lambda $ if and only if $\mu \in C^{\lambda, i \frown l} \backslash \{ \lambda,\lambda \cup \{i \frown k\} \mid i<k<l , k \notin re(\lambda)\} $ as desired.
\end{proof}
\vspace{1em}
\begin{defn}
For each $\mu \in C^{\lambda, i \frown l}$ define the \textit{shell coefficient} of $\lambda \cup \{i \frown l\}$ and $\mu$ as
\[c_\mu^{\lambda, i\frown l}=\displaystyle \frac{ t^{|(\lambda \cup \{i \frown l\})-\mu|} q^{\crs((\lambda \cup \{i \frown l\}) \cap \mu, (\lambda \cup \{i \frown l\})-\mu) }}{q^{\crs((\lambda \cup \{i \frown l\}) \cap \mu,\mu-(\lambda \cup \{i \frown l\}))}} \] 
where $t=q-1$ and $\crs(\cdot, \cdot)$ is the crossing number of two set partitions given in Section \ref{set partition}.
\end{defn}

\noindent We can associate each shell coefficient $c^{\lambda, i \frown l}_\mu$ to the shell created by the symmetric difference of $\lambda \cup \{i \frown l\}$ and $\mu$.  The next lemma shows the shell coefficient is the product of the shell coefficient of the outer whorl with the shell coefficient of the inner whorls.

\begin{lemma}\label{coefficient}
Let $\lambda \vdash [n]$, $i \notin le(\lambda)$, and $h \frown l, j \frown k \in \lambda$ with $1 \leq h < i<j <k < l \leq n$.  If $\mu \in C^{\lambda\mid_{j \mapsto i}, j \frown k}$ then
\[c^{\lambda, i \frown l}_\mu = c^{\lambda, i \frown l}_{\lambda \,\cup \{ i \frown k\}} c^{\lambda  \mid_{j \mapsto i}, j \frown k}_\mu. \]
\end{lemma}

\begin{proof}
Let $\mu \in C^{\lambda\mid_{j \mapsto i}, j \frown k}$.  By construction $i \notin le(\lambda|_{j \mapsto i})$, so $i \frown l \notin \mu$. Thus we have
$$(\lambda \cup \{i \frown l\}) -\mu =  \{ i \frown l\} \cup (\lambda-\mu),$$
hence
$$(\lambda \cup \{ i \frown l\}) \cap \mu = \lambda \cap \mu.$$
Substituting this and applying the crossing number equation (\ref{crossingidentity}), it follows that
\begin{eqnarray*}
c_\mu^{\lambda, i\frown l}&=&\displaystyle \frac{ t^{|(\lambda \cup \{i \frown l\})-\mu|} q^{\crs((\lambda \cup \{i \frown l\}) \cap \mu, (\lambda \cup \{i \frown l\})-\mu) }}{q^{\crs((\lambda \cup \{i \frown l\}) \cap \mu,\mu-(\lambda \cup \{i \frown l\}))}}\\ 
&=& \frac{t^{|\{i \frown l\} \cup (\lambda -\mu)|} q^{\crs(\lambda \cap \mu, \{i \frown l\} \cup (\lambda-\mu))} }{q^{\crs(\lambda \cap \mu, \mu - (\lambda \,\cup \{ i \frown l\}))}}\\ 
&=& \frac{t^{|i \frown l|} q^{\crs(\lambda \cap \mu, i \frown l)} t^{|\lambda -\mu|}q^{\crs(\lambda \cap \mu,\lambda-\mu)}}{q^{\crs(\lambda \cap \mu, \mu - (\lambda \,\cup \{ i \frown l\}))}}.
\end{eqnarray*}
Similarly since $j \frown k \in \lambda$ and $i \frown k \in \mu -\lambda$, we have
$$\mu -(\lambda \cup \{i \frown l\}) = \{i \frown k\} \cup ( \mu-(\lambda|_{j \mapsto i} \cup \{j \frown k\}))$$
and thus
$$\lambda-\mu =\lambda|_{j \mapsto i} \cup \{ j \frown k \}-\mu. $$
By the crossing number equation (\ref{crossingidentity}),
\begin{eqnarray*}
 c^{\lambda, i \frown l}_\mu &=& \frac{ t^{|i \frown l|} q^{\crs(\lambda\cap \mu, i \frown l)} t^{|(\lambda\mid_{j \mapsto i} \, \cup \, j \frown k)-\mu|}q^{\crs(\lambda \cap \mu,(\lambda\mid_{j \mapsto i} \, \cup \, j \frown k)-\mu)}}{q^{\crs(\lambda \cap \mu, \{i \frown k\} \cup (\mu - (\lambda|_{j \mapsto i} \, \cup \, j \frown k)))}}\\
 &=& \frac{t^{|i \frown l|} q^{\crs(\lambda\cap \mu, i \frown l)}}{q^{\crs(\lambda\cap\mu, i \frown k)}}\cdot \frac{ t^{|(\lambda\mid_{j \mapsto i} \, \cup \, j \frown k)-\mu|}q^{\crs(\lambda \cap \mu,(\lambda\mid_{j \mapsto i} \, \cup \, j \frown k)-\mu)}}{q^{\crs(\lambda \cap \mu, \mu - (\lambda|_{j \mapsto i} \, \cup \, j \frown k))}}.
 \end{eqnarray*}
Moreover any arc in $\lambda$ that crosses with $i \frown k$ or $i \frown l$ must be in $\mu$, implying
\begin{eqnarray*}
c^{\lambda, i \frown l}_\mu&=& \frac{t^{|i \frown l|} q^{\crs(\lambda, i \frown l)}}{q^{\crs(\lambda, i \frown k)}}\cdot \frac{ t^{|(\lambda\mid_{j \mapsto i} \, \cup \, j \frown k)-\mu|}q^{\crs((\lambda\mid_{j \mapsto i} \, \cup \, j \frown k) \cap \mu,(\lambda\mid_{j \mapsto i} \, \cup \, j \frown k)-\mu)}}{q^{\crs((\lambda\mid_{j \mapsto i} \, \cup \, j \frown k) \cap \mu, \mu - (\lambda\mid_{j \mapsto i} \, \cup \, j \frown k))}}\\
&=& c^{\lambda, i \frown l}_{\lambda \,\cup \{ i \frown k\}} c^{\lambda\mid_{j \mapsto i}, j \frown k}_\mu.
\end{eqnarray*}
\end{proof}

\begin{thm}\label{branchingrules}
For $\lambda \vdash [n]$, $i \notin le(\lambda)$, and $1\leq i < l \leq n$, we have
$$\chi^\lambda \odot \chi^{i \frownx l} = \sum_{\mu \in C^{\lambda, i \frown l}} c^{\lambda, i \frown l}_\mu \chi^\mu
\quad \text{where}\quad
 c_\mu^{\lambda, i\frown l}=\displaystyle \frac{ t^{|(\lambda \cup \{i \frown l\})-\mu|} q^{\crs((\lambda \cup \{i \frown l\}) \cap \mu, (\lambda \cup \{i \frown l\})-\mu) }}{q^{\crs((\lambda \cup \{i \frown l\}) \cap \mu,\mu-(\lambda \cup \{i \frown l\}))}} .
$$
where $C^{\lambda, i \frown l}$ is the shell set of $\lambda \cup \{i \frown l\}$ and $c^{\lambda, i \frown l}_\mu$ is the shell coefficient of $\lambda \cup \{i \frown l\}$ and $\mu$.

\end{thm}

\noindent Before proving the theorem we state a lemma about the $q$-analog of a crossing number.  In general, the $q$-analog of a nonnegative integer $n$ is 
\[[n]_q = \frac{q^n-1}{q-1}. \]

\begin{lemma}\label{crossing}
For $\lambda \vdash [n]$, and $1\leq j < l \leq n$ where $j \notin le(\lambda)$, we have
$$\sum_{\substack{i \frown k \in \lambda\\ i<j<k<l} }  q^{\crs(\lambda, j\frown k)} = \frac{q^{\crs(\lambda,j\frown l)}-1}{q-1} = [\crs(\lambda, j \frown k)]_q.$$
\end{lemma}

\begin{proof}
Let $\lambda \vdash [n]$, $1\leq j < l \leq n$, and $j \notin le(\lambda)$.  If the set of arcs in $\lambda$ that cross with $j\frown k$ is given by
\[ \{i \frown k \in \lambda\ \mid  i< j< k< l\} = \{i_1 \frown k_1, i_2 \frown k_2, \ldots, i_r \frown k_r\}, \]
then for $1 \leq s \leq r$
\[\{i \frown k \in \lambda\ \mid  i< j< k< k_s\} = \{i_1 \frown k_1, i_2 \frown k_2, \ldots, i_{s-1} \frown k_{s-1}\}. \] 
By the definition of the crossing number
\[ \sum_{\substack{i<j<k<l\\ i \frown k \in \lambda} }  q^{\crs(\lambda, j\frown k)} = \sum_{s=1}^r q^{\# \{ i \frown k \in \lambda \mid i < j < k <k_s \} }  = \sum_{s=1}^r q^{s-1} = \frac{q^r-1}{q-1} = \frac{q^{\crs(\lambda, j\frown l)}-1}{q-1}. \]

\end{proof}

\noindent We are now ready to prove Theorem \ref{branchingrules}.

\begin{proof}
We induct on $l-i$.  For the base case assume $l-i = 1$.  Then $C^{\lambda, i \frown l} = \{\lambda\}$, and we obtain
\[ \chi^\lambda \odot \chi^{i \frownx l} = \chi^\lambda \odot t \mathbbm{1} = t \chi^\lambda = c^{\lambda, i\frown l}_\lambda \chi^\lambda \]
as desired.

Assume the formula holds for all $k-j< l-i$.  Then, this yields
\begin{eqnarray*}
\chi^\lambda \odot \chi^{i \frownx l}&=& \chi^\lambda \odot t \bigg(\mathbbm{1}+\sum_{i<k<l} \chi^{i \frown k} \bigg)\\
&=& t \chi^\lambda + t \sum_{i<k<l} \chi^\lambda \odot \chi^{i \frown k}\\
&=& t\chi^\lambda + t \bigg( \sum_{\substack{i<k<l \\ k \notin re(\lambda)} } \chi^\lambda \odot \chi^{i \frown k} + \sum_{\substack{h<i<k<l\\ h \frown k \in \lambda} } \chi^\lambda \odot \chi^{i \frown k} + \sum_{\substack{i<j<k<l\\ j \frown k \in \lambda} } \chi^\lambda \odot \chi^{i \frown k} \bigg),
\end{eqnarray*}
which by Proposition \ref{tensorarcs} is equal to
\begin{eqnarray*}
&=& t\chi^\lambda +t \sum_{\substack{i<k<l \\ k \notin re(\lambda)}}\chi^{\lambda \cup  \{i \frown k\}}+ t\sum_{\substack{h<i<k<l\\ h \frown k \in \lambda}} \chi^\lambda \odot \chi^{i \frownx k}+ t\sum_{\substack{i<j<k<l\\ j \frown k \in \lambda}} \chi^{\lambda \cup  \{i \frown k\}- \{j\frown k\}} \odot \chi^{j \frownx k}.
\end{eqnarray*}
Recall from Definition \ref{replace} that $\lambda|_{j \mapsto i} =\lambda \cup \{i \frown k\} - \{j \frown k\}$ for each $j \frown k \in \lambda$ such that $i<j<k<l$. By the induction hypothesis the tensor product $\chi^\lambda \odot \chi^{i \frownx l}$ is
\begin{eqnarray*}
&=& t\chi^\lambda + t\hspace{-.1cm}\sum_{\substack{i<k<l \\ k \notin re(\lambda)}}\chi^{\lambda \cup  \{i \frown k\}}+ t\hspace{-.1cm}\sum_{\substack{h<i<k<l\\ h \frown k \in \lambda} } \bigg(\sum_{\substack{\mu \in C^{\lambda, i \frown k} } } c^{\lambda, i \frown k}_\mu \chi^\mu\bigg)+ t\hspace{-.1cm}\sum_{\substack{i<j<k<l\\ j \frown k \in \lambda}} \bigg(\sum_{\substack{\mu \in C^{\lambda|_{j \mapsto i}, j \frown k}} }\hspace{-.2cm} c^{\lambda|_{j \mapsto i}, j \frown k}_\mu \chi^\mu\bigg)\\
&=& t\chi^\lambda + t\sum_{\substack{i<k<l \\ k \notin re(\lambda)}}\chi^{ \lambda \cup  \{i \frown k\}}+ t \sum_{\mu \vdash [n] } \bigg(\sum_{\substack{h<i<k<l\\ h \frown k \in \lambda \\ \mu \in C^{\lambda, i \frown k} }} c^{\lambda, i \frown k}_\mu \chi^\mu\bigg)+ t\sum_{\mu \vdash [n] }\bigg(\sum_{\substack{i<j<k<l\\ j \frown k \in \lambda \\ \mu \in C^{\lambda|_{j \mapsto i}, j \frown k}}}\hspace{-.2cm}  c^{\lambda|_{j \mapsto i}, j \frown k}_\mu \chi^\mu\bigg).
\end{eqnarray*}

By Lemma \ref{crossing}, the coefficient of $\chi^\lambda$ will be
$$t + t \sum_{\substack{h<i<k<l\\ h \frown k \in \lambda}} c^{\lambda, i\frown k}_\lambda = t\bigg(1+  \sum_{\substack{h<i<k<l\\ h \frown k \in \lambda}} tq^{\crs(\lambda,i \frown k)}\bigg) = t \bigg(1 + t \cdot \frac{q^{\crs(\lambda,i \frown l)}-1}{t}\bigg) = t q^{\crs(\lambda,i \frown l)} = c^{\lambda, i \frown l}_\lambda.$$

\noindent Similarly for $\lambda \,\cup \{ i \frown k^\prime\}$ where $i < k^\prime < l$ and $k^\prime \notin re(\lambda)$, the coefficient of $\chi^{\lambda \,\cup \{ i \frown k^\prime\}}$ is
\[
t + t \sum_{\substack{h<i<k^\prime < k<l\\ h \frown k \in \lambda}} c^{\lambda, i\frown k}_\mu = t\bigg(1+  \sum_{\substack{h<i<k^\prime < k<l\\ h \frown k \in \lambda}} t\frac{q^{\crs(\lambda,i \frown k)}}{q^{\crs(\lambda,i \frown k^\prime)}}\bigg) = t\bigg(1+  \sum_{\substack{h<i < k<l\\ h \frown k \in \lambda}} tq^{\crs(\lambda - \nu,i \frown k)}\bigg),\]
where $\nu = \{h \frown k \in \lambda \mid (h \frown k, i \frown k^\prime) \in \Crs(\lambda, i \frown k^\prime)\}$. This is equivalent to
\[
 t \bigg(1 + t \cdot \frac{q^{\crs(\lambda-\nu,i \frown l)}-1}{t}\bigg)
= t \bigg(1 + t \cdot \frac{q^{\crs(\lambda,i \frown l)-\crs(\lambda,i\frown k^\prime)}-1}{t}\bigg)
= \frac{t q^{\crs(\lambda,i \frown l)}}{q^{\crs(\lambda, i\frown k^\prime)}} = c^{\lambda, i \frown l}_{\lambda \,\cup \{ i \frown k\}}
\]
by Lemma \ref{crossing}.  If $j \frown k^\prime \in \lambda$ is such that $i < j < k^\prime < l$ then we have $\lambda|_{j \mapsto i} =\lambda \cup \{i \frown k^\prime\} - \{j \frown k^\prime\}$.  Let $\nu = \{h \frown k \in \lambda \mid (h \frown k, i \frown k^\prime) \in \Crs(\lambda, i \frown k^\prime)\}$.  Using Lemma \ref{coefficient}, the coefficient of $\chi^\mu$ for each $\mu \in C^{\lambda|_{j \mapsto i}, j \frown k^\prime}$ is
\begin{eqnarray*}
 t c^{\lambda|_{j \mapsto i}, j \frown k^\prime}_\mu + t \sum_{\substack{h<i< k^\prime <k<l\\ h \frown k \in \lambda}} c^{\lambda, i\frown k}_\mu 
 &=& t c^{\lambda|_{j \mapsto i}, j \frown k^\prime}_\mu + t \sum_{\substack{h<i< k^\prime <k<l\\ h \frown k \in \lambda}} c^{\lambda, i \frown k}_{\lambda\cup i \frown k^\prime} c^{\lambda|_{j \mapsto i}, j \frown k^\prime}_\mu \\
&=& tc^{\lambda|_{j \mapsto i}, j \frown k^\prime}_\mu \bigg(1+ \sum_{\substack{h<i<k^\prime <k<l\\ h \frown k \in \lambda}} t \frac{q^{\crs(\lambda,i \frown k)}}{q^{\crs(\lambda ,i \frown k^\prime)}} \bigg)\\
&=& tc^{\lambda|_{j \mapsto i}, j \frown k^\prime} \bigg(1+ t \cdot \frac{q^{\crs(\lambda -\nu,i \frown l)}-1}{t} \bigg).
\end{eqnarray*}
Applying Lemmas \ref{crossing} and \ref{coefficient} yields
\begin{eqnarray*}
 tc^{\lambda|_{j \mapsto i}, j \frown k^\prime} \bigg(1+ t \cdot \frac{q^{\crs(\lambda -\nu,i \frown l)}-1}{t} \bigg) 
&=& tc^{\lambda|_{j \mapsto i}, j \frown k^\prime} \bigg(1+ t \cdot \frac{q^{\crs(\lambda ,i \frown l)- \crs(\lambda,i\frown k^\prime)}-1}{t} \bigg) \\
&=& \frac{t q^{\crs(\lambda,i \frown l)}}{q^{\crs(\lambda,i\frown k^\prime)}}c^{\lambda|_{j \mapsto i}, j \frown k^\prime}_\mu\\
&=& c^{\lambda, i \frown l}_{\lambda \,\cup \{ i \frown k\}^\prime} c^{\lambda \mid_{j \mapsto i}, j\frown k^\prime}_\mu\\
&=& c^{\lambda, i \frown l}_\mu.
\end{eqnarray*}

Substituting this into the equation for $\chi^\lambda \odot \chi^{i\frownx l}$ and applying Lemma \ref{shellset} we obtain
\begin{eqnarray*}
\chi^\lambda \odot \chi^{i\frownx l} &=& c^{\lambda, i \frown l}_\lambda \chi^\lambda + \sum_{\substack{i<k<l \\ k \notin re(\lambda)}} c^{\lambda, i \frown l}_{\lambda \,\cup \{ i \frown k\}} \chi^{\lambda \,\cup \{ i \frown k\}}+   \sum_{\substack{\mu \vdash [n] } } \sum_{\substack{i<j<k<l\\ j \frown k \in \lambda \\ \mu \in C^{\lambda|_{j \mapsto i}, j \frown k}}}c^{\lambda, i \frown l}_\mu \chi^\mu \\
&=& \sum_{\mu \in C^{\lambda, i \frown l}} c^{\lambda, i \frown l}_\mu \chi^\mu .
\end{eqnarray*}

\end{proof}

This combinatorial description of the coefficients in the tensor product leads to a combinatorial description of the coefficients in the restriction to $U_{n-1}$.

\begin{cor}\label{restriction}
For $\lambda \vdash [n]$, the restriction $\Res_{U_{n-1}}^{U_n} (\chi^\lambda)$ is given by
$$\Res_{U_{n-1}}^{U_n} (\chi^\lambda) = \sum_{\mu\, \vdash [n-1]} c_\mu^\lambda \chi^{\mu}$$
where
\begin{equation*}
c_\mu^\lambda =\left\{\begin{array}{ll}
\delta_{\lambda \mu} &  \text{if}\ n\notin re(\lambda), \\
\displaystyle \frac{ t^{|\lambda-\mu|} q^{\crs(\lambda \cap \mu, \lambda-\mu) }}{q^{\crs(\lambda \cap \mu,\mu-\lambda)}} & \text{if}\ \mu \in C^{\lambda- \{i \frown n\}, i \frown n}, \\
 0 & \text{otherwise.}\end{array}\right.
\end{equation*}
\end{cor}

\begin{proof}
Applying Propositions \ref{restrictionarc}, \ref{tensorarcs}, and Theorem \ref{branchingrules} respectively,
\begin{eqnarray*}
\Res^{U_n}_{U_{n-1}}(\chi^\lambda) &=& \bigodot_{i \frown l \in \lambda} \Res^{U_n}_{U_{n-1}} (\chi^{i\frown l}) =\bigodot_{\substack{j \frown l \in \lambda \\ l \neq n}} \chi^{j\frown l} \odot \chi^{i \frownx n}\\
& =&  \chi^{\lambda-\{i\frown n\}} \odot \chi^{i \frownx n} = \sum_{\mu \in C^{\lambda- \{i \frown n\}, i \frown n}} c^\lambda_\mu \chi^\mu
 \end{eqnarray*}
where
\[ c^\lambda_\mu = c^{\lambda - \{i \frown n\}, i \frown n}_\mu =  \frac{ t^{|\lambda-\mu|} q^{\crs(\lambda \cap \mu, \lambda-\mu) }}{q^{\crs(\lambda \cap \mu,\mu-\lambda)}}. \]
\end{proof}

\begin{example}
Similar to Example \ref{seashells}, let
\[\lambda =
\begin{tikzpicture}
\draw[fill=black] (-1,0) circle (.07);
 \draw[fill=black] (-.5,0) circle (.07);
\draw[fill=black] (.0,0) circle (.07);
\draw[fill=black] (.5,0) circle (.07);
\draw[fill=black] (1,0) circle (.07);
\draw[fill=black] (1.5,0) circle (.07);
\draw[line width=.2mm] (-1,0) to[in=110, out=70, circle, looseness=1.4] (.5,0);
\draw[line width=.2mm] (-.5,0) to[in=110, out=70, circle, looseness=1.35] (1.5,0);
\draw[line width=.2mm] (0,0) to[in=110, out=70, circle, looseness=1.4] (1,0);
\end{tikzpicture}
\]
so that
\[C^{\lambda-\{2\frown 6\}, 2\frown 6} =  \left\{
\begin{tikzpicture}
\draw[fill=black] (-1,0) circle (.07);
 \draw[fill=black] (-.5,0) circle (.07);
\draw[fill=black] (.0,0) circle (.07);
\draw[fill=black] (.5,0) circle (.07);
\draw[fill=black] (1,0) circle (.07);
\draw[line width=.2mm] (-1,0) to[in=110, out=70, circle, looseness=1.4] (.5,0);
\draw[line width=.2mm] (0,0) to[in=110, out=70, circle, looseness=1.4] (1,0);
\end{tikzpicture} \quad,\quad
\begin{tikzpicture}
\draw[fill=black] (-1,0) circle (.07);
 \draw[fill=black] (-.5,0) circle (.07);
\draw[fill=black] (.0,0) circle (.07);
\draw[fill=black] (.5,0) circle (.07);
\draw[fill=black] (1,0) circle (.07);
\draw[line width=.2mm] (-1,0) to[in=110, out=70, circle, looseness=1.4] (.5,0);
\draw[line width=.2mm] (0,0) to[in=110, out=70, circle, looseness=1.4] (1,0);
\draw[line width=.2mm] (-.5,0) to[in=110, out=70, circle, looseness=1.7] (0,0);
\end{tikzpicture} \quad,\quad
\begin{tikzpicture}
\draw[fill=black] (-1,0) circle (.07);
 \draw[fill=black] (-.5,0) circle (.07);
\draw[fill=black] (.0,0) circle (.07);
\draw[fill=black] (.5,0) circle (.07);
\draw[fill=black] (1,0) circle (.07);
\draw[line width=.2mm] (-1,0) to[in=110, out=70, circle, looseness=1.4] (.5,0);
\draw[line width=.2mm] (-.5,0) to[in=110, out=70, circle, looseness=1.4] (1,0);
\end{tikzpicture}
\right\}.
\]
Drawing the arcs of $\mu= \{1\frown 4, 2\frown 3 \frown 5\}$ below the nodes of $\lambda$ as shown below
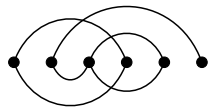
\begin{center}
\begin{tikzpicture}
\draw[fill=black] (-1,0) circle (.07);
 \draw[fill=black] (-.5,0) circle (.07);
\draw[fill=black] (.0,0) circle (.07);
\draw[fill=black] (.5,0) circle (.07);
\draw[fill=black] (1,0) circle (.07);
\draw[fill=black] (1.5,0) circle (.07);
\draw[line width=.2mm] (-1,0) to[in=110, out=70, circle, looseness=1.4] (.5,0);
\draw[line width=.2mm] (-1,0) to[in=-110, out=-70, circle, looseness=1.4] (.5,0);
\draw[line width=.2mm] (-.5,0) to[in=110, out=70, circle, looseness=1.35] (1.5,0);
\draw[line width=.2mm] (0,0) to[in=110, out=70, circle, looseness=1.4] (1,0);
\draw[line width=.2mm] (0,0) to[in=-110, out=-70, circle, looseness=1.4] (1,0);
\draw[line width=.2mm] (-.5,0) to[in=-110, out=-70, circle, looseness=1.7] (0,0);
\end{tikzpicture}
\end{center}
illustrates that \[c^\lambda_\mu = \frac{t^1 \cdot q^1}{q^0} = tq\] since
\[ \lambda - \mu = \{2 \frown 6\},\quad \Crs(\lambda \cap \mu, \lambda-\mu) = \{(1\frown 4, 2\frown 6) \}, \quad \crs(\lambda \cap \mu, \lambda-\mu) = 1, \quad \crs(\lambda \cap \mu, \mu-\lambda) = 0 .\]
We can calculate the other coefficients in the same manner to obtain
\[\Res^{U_6}_{U_{5}}(\chi^\lambda) = tq\chi^{\{1\frown 4, 3\frown 5 \}}+ tq\chi^{\{2 \frown 3\frown 5, 1 \frown 4  \}}+ t^2q\chi^{\{1 \frown 4, 2 \frown 5\}}.\]
\end{example}

\subsection{Induction and Superinduction}
While the restriction of a supercharacter of $U_n$ is a nonnegative integer linear combination of supercharacters, an induced supercharacter may not be a sum of supercharacters.  In fact, the induced character may not even be a superclass function; for an example see \cite[Section 6]{supercharacter}. If instead we generalize to superinduction by averaging over superclasses in the same way that induction averages over conjugacy classes, then the constructed function will be a linear combination of supercharacters with rational coefficients \cite[Lemma 6.7]{supercharacter}.

Suppose $H\subseteq G$ and $\chi$ is a superclass function of $H$.  If $\mathcal{K}_g$ is the superclass containing $g \in G$, then the \textit{superinduction} $\SInd^G_H(\chi)$ is
\[ \SInd^{G}_H(\chi) (g) = |G:H| \frac{1}{|\mathcal{K}_g|} \sum_{x \in \mathcal{K}_g} \dot{\chi}(x) \ \text{where}\ \dot{\chi}(x) = \left\{ \begin{array}{lr} \chi(x)& \text{if}\ x\in H \\ 0 & \text{if}\ x\not\in H. \end{array}\right. \]
A nice property of superinduction is that the analog of Frobenius reciprocity holds.

\begin{proposition}[{Frobenius Reciprocity \cite[Lemma 5.2]{hendrickson}}]\label{superinduction frobenius reciprocity}
Let $H$ be a subgroup of $G$.  Suppose $\varphi$ is a superclass function of $G$ and $\theta$ is a class function of $H$.  Then
\[ \langle \SInd_H^G(\theta) , \varphi \rangle_G = \langle \theta, \Res_H^G (\varphi) \rangle_H. \]
\end{proposition}

However, superinduced characters are not necessarily characters so it is useful to know when superinduction is equivalent to induction.

\cite[Section 3.2]{superinduction} examines some cases when this occurs for a larger class of $p$-groups known as algebra groups.  If $J$ is a finite dimensional nilpotent associative algebra over $\mathbb{F}_q$, then the algebra group based on $J$ is $G = \{1+x \mid x \in J\}$ under the multiplication $(1+x)(1+y) = 1+x+y +xy$. In particular, Marberg and Thiem show if we embed $U_{n-1}$ into $U_n$ by
\[ U_{n-1} = \{u \in U_n \mid (u-1)_{ij} \neq 0 \ \text{implies}\ i<j<n \} \]
then for any superclass function $\chi$ of $U_{n-1}$,
\[ \SInd_{U_{n-1}}^{U_n}(\chi) = \Ind_{U_{n-1}}^{U_n}(\chi). \]
They also provide some conditions when superinduction is the same as induction.

\begin{proposition}[{\cite[Theorem 3.1]{superinduction}}]\label{superinduction vs induction}
Let $H$ be a subalgebra group of an algebra group $G$, and suppose

\begin{enumerate}
\item no two superclasses of $H$ are in the same superclass of $G$, and
\item $x(h-1)+1 \in H$ for all $x \in G, h \in H$.
\end{enumerate}
Then the superinduction of any superclass function $\chi$ of $H$ is
$$\SInd_H^G(\chi) = \Ind_H^G(\chi).$$
\end{proposition}
 
If we embed $U_{n-1}$ into $U_n$ by 
\[U_{n-1} = \{u \in U_n \mid u_{n-1,n} = 0 \ \text{and}\ u_{i,n-1} = 0 \ \text{for}\ i < n-1 \} \] 
then we have the following corollary.
 
\begin{cor}
Let $U_{n-1} = \{u \in U_n \mid u_{n-1,n} = 0 \ \text{and}\ u_{i,n-1} = 0 \ \text{for}\ i < n-1 \}$.  Then the superinduction any superclass function $\chi$ of $U_{n-1}$ is
$$\SInd_{U_{n-1}}^{U_n}(\chi) = \Ind_{U_{n-1}}^{U_n}(\chi).$$
\end{cor}

\begin{proof}
It suffices to show the hypotheses of the previous theorem hold.  Because there is an injective function from superclasses of $U_{n-1}$ to $U_n$ then no two superclasses of $U_{n-1}$ are in the same superclass of $U_n$.

Let $x \in U_n$, $h \in U_{n-1}$ and $u = x(h-1) +1$.  Since $h_{i, n-1}-1 = 0$ we have $u_{i, n-1} = 0$ for $i < n-1$.  Similarly, $x_{n-1, j} = 0 $ for $j < n-1$ and $h_{n-1, j}-1 = 0$ for $j \geq n-1$ implies $u_{n-1,n} = 0$.  This shows $u \in U_{n-1}$.  Therefore, $\text{SInd}_{U_{n-1}}^{U_n}(\chi) = \text{Ind}_{U_{n-1}}^{U_n}(\chi)$ for any superclass function $\chi$ of $U_{n-1}$ by Proposition \ref{superinduction vs induction}.
\end{proof}

Unlike in the representation theory of the symmetric group, the decomposition of induced characters depends on the embedding of $U_{n-1}$ into $U_n$.  If we instead consider right modules, then superinduction is equivalent to induction for the following embeddings
\[U_{n-1} = \{u \in U_n \mid (u-1)_{ij} \neq 0 \ \text{implies} \ 1<i<j  \} \] 
and
\[U_{n-1} = \{u \in U_n \mid u_{1,2} = 0 \ \text{and}\ u_{2,j} = 0 \ \text{for}\ 2 < j  \} \] 
\cite[Section 3.1]{superinduction}.  However, it is not known if superinduction is the same as induction for other embeddings.  In our case we use the embedding of $U_{n-1}\subseteq U_n$ obtained by removing the last column so that superinduction is in fact induction.  

We now derive a corresponding formula for induction from restriction.

\begin{cor}\label{induction}
For $\mu \vdash [n-1]$, the induction $\Ind_{U_{n-1}}^{U_n} (\chi^\mu)$ is given by
\[\Ind_{U_{n-1}}^{U_n} (\chi^\mu) = \sum_{\lambda \, \vdash [n]} d_\mu^\lambda \chi^{\lambda},\]
where
\begin{equation*}
d_\mu^\lambda =\left\{\begin{array}{ll}
\delta_{\lambda \mu} &  \text{if}\ n\notin re(\lambda) \\
\displaystyle \frac{ t^{|\mu-\lambda|} q^{\crs(\mu-\lambda, \lambda \cap \mu) }}{q^{\crs(\lambda- \mu,\lambda \cap \mu)}} & \text{if}\ \mu \in C^{\lambda- \{i \frown n\}, i \frown n} \\
 0 & \text{otherwise.}\end{array}\right.
\end{equation*}
\end{cor}

\begin{proof}

Let $\lambda \vdash [n]$ and $\mu \vdash [n-1]$.  Frobenius reciprocity, Proposition \ref{superinduction frobenius reciprocity}, shows
$$\langle \chi^\lambda, \SInd_{U_{n-1}}^{U_n} (\chi^\mu)\rangle_{U_n} = \langle \Res_{U_{n-1}}^{U_n}(\chi^\lambda), \chi^\mu \rangle_{U_{n-1}}.$$
Thus if
$$\Ind_{U_{n-1}}^{U_n} (\chi^\mu) = \sum_\gamma d^\gamma_\mu \chi^\gamma \quad \text{and} \quad  \Res_{U_{n-1}}^{U_n}(\chi^\lambda) = \sum_\nu c^\lambda_\nu \chi^\nu$$
then the inner product, Proposition \ref{superorthogonality}, yields
$$q^{\crs(\lambda,\lambda)} t^{|\lambda|} d^\lambda_\mu = q^{\crs(\mu,\mu)} t^{|\mu|} c^\lambda_\mu.$$
Therefore, the coefficient $d^\lambda_\mu$ is
\begin{eqnarray*}
d^\lambda_\mu &=&\frac{t^{|\mu|-|\lambda|} q^{\crs(\mu,\mu)}}{q^{\crs(\lambda,\lambda)}} c^\lambda_\mu \\
&=&\frac{t^{|\mu|-|\lambda|} q^{\crs(\mu,\mu)}}{q^{\crs(\lambda,\lambda)}}\cdot \frac{ t^{|\lambda-\mu|} q^{\crs(\lambda \cap \mu, \lambda-\mu) }}{q^{\crs(\lambda \cap \mu,\mu-\lambda)}}\\
&=&\frac{t^{|\mu-\lambda|} q^{\crs(\mu,\mu)-\crs(\lambda \cap \mu, \mu-\lambda)}}{q^{\crs(\lambda,\lambda)-\crs(\lambda \cap \mu, \lambda-\mu)}}
\end{eqnarray*}
since $|\mu-\lambda| = |\mu|-|\lambda|+|\lambda-\mu|$.  From the crossing number equation (\ref{crossingidentity}) we obtain

\begin{eqnarray*}
d^\lambda_\mu &=&\frac{t^{|\mu-\lambda|} q^{\crs(\mu-(\lambda \cap \mu),\mu- (\mu-\lambda))}}{q^{\crs(\lambda-(\lambda \cap \mu),\lambda-(\lambda-\mu))}}\\
&=&   \frac{ t^{|\mu-\lambda|} q^{\crs(\mu-\lambda, \lambda \cap \mu) }}{q^{\crs(\lambda- \mu,\lambda \cap \mu)}}.
\end{eqnarray*}\end{proof}

Together Corollaries \ref{restriction} and \ref{induction} for decomposing restricted and induced supercharacters are known as \textit{branching rules}, which we restate due to their importance.

\begin{thm}[Branching Rules]\label{branching rules}
For $\lambda \vdash [n]$, the restriction $\Res_{U_{n-1}}^{U_n} (\chi^\lambda)$ is given by
$$\Res_{U_{n-1}}^{U_n} (\chi^\lambda) = \sum_{\mu\, \vdash [n-1]} c_\mu^\lambda \chi^{\mu}$$
where
\begin{equation*}
c_\mu^\lambda =\left\{\begin{array}{ll}
\delta_{\lambda \mu} &  \text{if}\ n\notin re(\lambda), \\
\displaystyle \frac{ t^{|\lambda-\mu|} q^{\crs(\lambda \cap \mu, \lambda-\mu) }}{q^{\crs(\lambda \cap \mu,\mu-\lambda)}} & \text{if}\ \mu \in C^{\lambda- \{i \frown n\}, i \frown n}, \\
 0 & \text{otherwise.}\end{array}\right.
 \end{equation*}
 For $\mu \vdash [n-1]$, the induction $\Ind_{U_{n-1}}^{U_n} (\chi^\mu)$ is given by
\[\Ind_{U_{n-1}}^{U_n} (\chi^\mu) = \sum_{\lambda \, \vdash [n]} d_\mu^\lambda \chi^{\lambda},\]
where
\begin{equation*}
d_\mu^\lambda =\left\{\begin{array}{ll}
\delta_{\lambda \mu} &  \text{if}\ n\notin re(\lambda) \\
\displaystyle \frac{ t^{|\mu-\lambda|} q^{\crs(\mu-\lambda, \lambda \cap \mu) }}{q^{\crs(\lambda- \mu,\lambda \cap \mu)}} & \text{if}\ \mu \in C^{\lambda- \{i \frown n\}, i \frown n} \\
 0 & \text{otherwise.}\end{array}\right.
\end{equation*}
\end{thm}

\noindent Since the branching rules are simple and easily computable, we are submiting code for a program in SageMath that takes a supercharacter as input and outputs these restriction and induction decompositions.  This enables us to quickly compute meaningful examples of restricting and inducing a supercharacter multiple times.  While these formulas allow us to better understand restriction and induction, they are also useful for Schur--Weyl duality.

\section{Shell Tableaux}

We use the branching rules to create a graph known as the Bratteli diagram.  For the symmetric group, paths in the Bratteli diagram are indexed by a set of combinatorial objects called Young tableaux \cite{Mac}.  Building from the combinatorics of the previous section, we create an analog of Young tableaux known as shell tableaux and construct a bijection between shell tableaux and paths in the Bratteli diagram.  

For $k \in \mathbb{Z}_{\geq1}$, consider
\[V^{ k} = \underbrace{(\Ind_{U_{n-1}}^{U_n} \Res_{U_{n-1}}^{U_n})}_{k{\rm \  times}}\hspace{-1mm}{}^k(\mathbbm{1}) \]
where $\mathbbm{1}$ is the trivial supercharacter of $U_n$ that is restricted and induced $k$ times.  This is reminiscent of the situation in the partition algebra where the permutation representation of the symmetric group is isomorphic to restricting and then inducing the trivial character.  Note that if $V= \mathbb{C}U_n \otimes_{\mathbb{C}U_{n-1}} \mathbbm{1}$ then \[V= \mathbb{C}U_n \otimes_{\mathbb{C}U_{n-1}} \Res^{U_n}_{U_{n-1}}(\mathbbm{1}) = \Ind_{U_{n-1}}^{U_n} \Res_{U_{n-1}}^{U_n}(\mathbbm{1}) \] 
by the definition of induction.  More broadly, we have the following generalization of the tensor identity from \cite[(3.18)]{Partition}.

\begin{lemma}\label{isomorphism}
Let $H$ be a subgroup of a group $G$.  For a $G$-module $M$, the map
\begin{eqnarray*}
 \tau: \mathbb{C}G \otimes_{\mathbb{C}H} \Res^G_H(M) &\longrightarrow& (\mathbb{C}G\otimes_{\mathbb{C}H} \mathbbm{1}) \otimes M\\
 g \otimes m &\mapsto& (g \otimes \mathbbm{1} ) \otimes gm\\
 g \otimes g^{-1}m &\mapsfrom& (g \otimes \mathbbm{1}) \otimes m
 \end{eqnarray*}
 is a $G$-module isomorphism.
\end{lemma}
\noindent Iterating this identity, we obtain
\[V^k \cong V^{\otimes k}. \] 

Let
\begin{eqnarray*}
 \hat{Z}_k &=& \big\{\lambda \vdash [n] \mid \Hom_{U_n}(U_n^\lambda, V^k) \neq \emptyset \big\} \\
\hat{Z}_{k+\frac{1}{2}} &=& \big\{\mu \vdash [n-1]\mid \Hom_{U_{n-1}}(U_{n-1}^\mu, \Res_{U_{n-1}}^{U_n}(V^{ k}))\neq \emptyset \big\}.
\end{eqnarray*}
The \textit{Bratteli diagram} $\Lambda(n)$ is the graph with
\begin{enumerate}[label={(\alph*)}]
\item vertices $\{(\lambda, k) \mid k \in \mathbb{Z}_{\geq 0}, \lambda \in \hat{Z}_k \} \cup \{(\mu, k+\frac{1}{2}) \mid k \in \mathbb{Z}_{\geq 0}, \mu \in \hat{Z}_{k+\frac{1}{2}} \}$,
\item an edge $(\lambda,k) \to (\mu ,k+\frac{1}{2})$ if $\langle \Res^{U_n}_{U_{n-1}}(\chi^\lambda), \chi^\mu \rangle \neq 0$,
\item an edge $(\mu,k+\frac{1}{2}) \to (\lambda ,k+1)$ if $\langle \chi^\lambda, \Ind^{U_n}_{U_{n-1}}(\chi^\mu) \rangle \neq 0$ ,
\item an edge labeling $m: E \to \mathbb{Z}_{\geq 1}$ on the set of edges $E$ defined by
\[\begin{array}{ccc}
m((\lambda, k)\to (\mu, k+\frac{1}{2})) &=& \displaystyle{\frac{(q-1)^{|\lambda-\mu|}q^{\crs(\lambda \cap \mu, \lambda-\mu)} }{q^{\crs(\lambda \cap \mu, \mu-\lambda)}}}\\
&&\vspace{-.1in}\\
 m((\mu, k+\frac{1}{2})\to (\lambda, k+1)) &=&  \displaystyle{\frac{(q-1)^{|\mu-\lambda|}q^{\crs(\mu-\lambda,\lambda \cap \mu)} }{q^{\crs(\lambda - \mu, \lambda \cap \mu)}}}.
\end{array}\]

\end{enumerate}
Recall from the branching rules, Theorem \ref{branching rules}, that the edge labeling $m((\lambda, k)\to (\mu, k+\frac{1}{2}))$ is the restriction coefficient which specifies the multiplicity that $\chi^\mu$ appears in $\Res_{U_{n-1}}^{U_n}(\chi^\lambda)$.  Similarly, the edge labeling $m((\lambda, k+\frac{1}{2})\to (\mu, k+1))$ is the induction coefficient which specifies the multiplicity that $\chi^\lambda$ appears in $\Ind_{U_{n-1}}^{U_n}(\chi^\mu)$.

When drawing the Bratteli diagram, we place all the vertices $(\lambda, l)$ in the $l$th row and simply write $\lambda$.  For example, the Bratteli diagram for $\Lambda(3)$ up to row $3$ is 
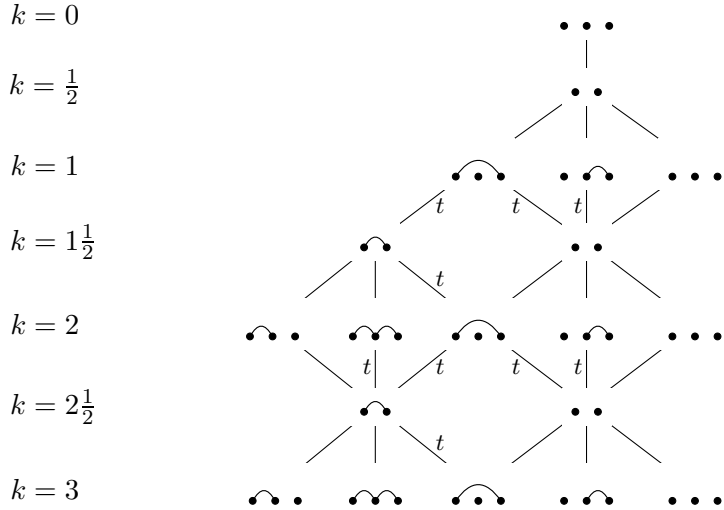
\[\hspace{-1.2in}
\centerline{\begin{tikzpicture}[xscale=3*.12, yscale=-4.5*.27]
\begin{scope}[black]
	\coordinate (0) at (0,.4); \coordinate (00) at (-20,.4);
	\coordinate (11) at (0,1.2); \coordinate(10) at (-20,1.2);
	\coordinate (31) at (-4, 2.05);\coordinate (32) at (0, 2.05);\coordinate (33) at (4,2.05);
	\coordinate(30) at (-20,2.05); \coordinate(50) at (-20,2.9);
	\coordinate (51) at (-7.8,2.9); \coordinate (52) at (0,2.9);
	\coordinate (71) at (-11.6,3.8); \coordinate (72) at (-7.8,3.8); \coordinate (73) at (0,3.8);
	\coordinate (74) at (-4,3.8); \coordinate (75) at (4,3.8);
	\coordinate(70) at (-20,3.8); \coordinate (80) at (-20,4.7); \coordinate (90) at (-20, 5.6);
	\coordinate (81) at (-7.8,4.7); \coordinate (82) at (0,4.7);
	\coordinate (91) at (-11.5,5.6); \coordinate (92) at (-7.8,5.6); \coordinate (93) at (0,5.6);
	\coordinate (94) at (-4,5.6); \coordinate (95) at (4,5.6);
	\node at (-5.4, 2.5) {\footnotesize{$t$}};
	\node at (-.3, 2.5) {\footnotesize{$t$}};
	\node at (-2.6, 2.5) {\footnotesize{$t$}};
	\node at (-5.4, 3.3) {\footnotesize{$t$}};
	\node at (-5.4, 4.26) {\footnotesize{$t$}};
	\node at (-.3, 4.26) {\footnotesize{$t$}};
	\node at (-2.6, 4.26) {\footnotesize{$t$}};
	\node at (-8.1, 4.26) {\footnotesize{$t$}};
	\node at (-5.4, 5.1) {\footnotesize{$t$}};
\draw (0)--(11);
\draw (11)--(31) (11)--(32) (11)--(33);
\draw (31)--(51) (31)--(52) (32)--(52) (33)--(52);
\draw (51)--(71) (51)--(72) (51)--(74) (52)--(74) (52)--(75)  (52)--(73);
\draw (71)--(81) (73)--(82) (72)--(81) (74)--(81) (74)--(82) (75)--(82);
\draw (82)--(93) (82)--(94) (82)--(95) (81)--(91) (81)--(92) (81)--(94);
\begin{scope}[every node/.style={fill=white}]
	\node at (0) {\Arc};
	\node at (11) {\arc};
	\node at (32) {\ArC};
	\node at (33) {\Arc};
	\node at (31) {\aRC};
	\node at (52) {\arc};
	\node at (51) {\aRc};
	\node at (75) {\Arc};
	\node at (72) {\ARc};
	\node at (73) {\ArC};
	\node at (74) {\aRC};
	\node at (71) {\ARC};
	\node at (82) {\arc};
	\node at (81) {\aRc};
	\node at (95) {\Arc};
	\node at (92) {\ARc};
	\node at (93) {\ArC};
	\node at (94) {\aRC};
	\node at (91) {\ARC};
	\node at (00) {$\scriptsize{k=0}$};
	\node at (10) {$\scriptsize{k=\frac{1}{2}}$};
	\node at (30) {$\scriptsize{k=1}$};
	\node at (50) {\ $\scriptsize{\ k=1\frac{1}{2}}$};
	\node at (70) {$\scriptsize{k=2}$};
	\node at (80) {\ $\scriptsize{\ k=2\frac{1}{2}}$};
	\node at (90) {$\scriptsize{k=3}$};
\end{scope}
\end{scope} 
\end{tikzpicture}}
\]
where $t = q-1$.

A \textit{path} $P$ in the Bratteli diagram $\Lambda(n)$ to $\lambda \in \hat{Z}_k$ is a sequence $P= (\lambda^0, \lambda^{\frac{1}{2}},\ldots, \lambda^{k-\frac{1}{2}}, \lambda^k = \lambda)$ such that for $0 \leq r \leq k-1$,
\begin{enumerate}[label={(\alph*)}]
\item $(\lambda^r , r)$ and $(\lambda^{r+\frac{1}{2}}, r+\frac{1}{2})$ are vertices in $\Lambda(n)$
\item $(\lambda^r,r)\to (\lambda^{r+\frac{1}{2}}, r+\frac{1}{2})$ and $(\lambda^{r+\frac{1}{2}}, r+\frac{1}{2}) \to (\lambda^{r+1}, r+1)$ are edges in $\Lambda(n)$.
\end{enumerate}
For instance, \[\begin{tikzpicture}[scale=1.5]
\node at (-.6,0) {$P = \Big($};
  \draw[fill=black] (0,0) circle (.04);
\draw[fill=black] (.25,0) circle (.04);
\draw[fill=black] (.5,0) circle (.04);
\node at (.75,-.1) {,};
  \draw[fill=black] (1,0) circle (.04);
\draw[fill=black] (1.25,0) circle (.04);
\node at (1.5,-.1) {,};
  \draw[fill=black] (1.75,0) circle (.04);
\draw[fill=black] (2,0) circle (.04);
\draw[fill=black] (2.25,0) circle (.04);
\draw[line width=.15mm] (1.75,0) to[in=120, out=60, circle, looseness=1.5] (2.25,0);
\node at (2.5,-.1) {,};
  \draw[fill=black] (2.75,0) circle (.04);
\draw[fill=black] (3,0) circle (.04);
\draw[line width=.15mm] (2.75,0) to[in=120, out=60, circle, looseness=1.8] (3,0);
\node at (3.25,-.1) {,};
\draw[fill=black] (3.5,0) circle (.04);
\draw[fill=black] (3.75,0) circle (.04);
\draw[fill=black] (4,0) circle (.04);
\draw[line width=.15mm] (3.5,0) to[in=120, out=60, circle, looseness=1.8] (3.75,0);
\draw[line width=.15mm] (3.75,0) to[in=120, out=60, circle, looseness=1.8] (4,0);
\node at (4.25,-.1) {,};
\draw[fill=black] (4.5,0) circle (.04);
\draw[fill=black] (4.75,0) circle (.04);
\draw[line width=.15mm] (4.5,0) to[in=120, out=60, circle, looseness=1.8] (4.75,0);
\node at (5,-.1) {,};
\draw[fill=black] (5.25,0) circle (.04);
\draw[fill=black] (5.5,0) circle (.04);
\draw[fill=black] (5.75,0) circle (.04);
\draw[line width=.15mm] (5.25,0) to[in=120, out=60, circle, looseness=1.8] (5.5,0);
\node at (6,0) {\Big)};
\end{tikzpicture}\]
is a path in $\Lambda(3)$.

Taking the edge labeling into account, we say the \textit{weight} $\wt(P)$ of a path $P$ is the product 
\[\prod_{r=1}^{k-1} m((\lambda^r,r)\to (\lambda^{r+\tfrac{1}{2}}, r+\tfrac{1}{2}))m((\lambda^{r+\frac{1}{2}},r+\tfrac{1}{2})\to (\lambda^{r+1}, r+1))\]
of its edge labels.  The sum of the weights of the paths to $\lambda\in \hat{Z}_k$ is the multiplicity that $\chi^\lambda$ appears in $V^k$.  The path given above has weight $t^2$ since $m((\lambda^1,1) \to (\lambda^{1\frac{1}{2}}, \frac{1}{2})) = t$ and $m((\lambda^2,2) \to (\lambda^{2 \frac{1}{2}}, \frac{1}{2})) = t$.

Let $\mathcal{P}_k(\lambda)$ be the set of paths in $\Lambda(n)$ to $\lambda\in \hat{Z}_k$.  There is a combinatorial way to encode paths in $\mathcal{P}_k(\lambda)$ using a generalization of shells.

\begin{defn}\label{generalizedshell}
Let $s^\prime \in \{ s, s+1\}$ for $s \in \mathbb{Z}_{\geq 1}$ and $1 \leq i \leq l \leq n$.  A \textit{generalized shell} of \textit{width} $l-i$ is a set of arcs on $n$ nodes of the form
\[ \bigcup_{r=1}^s \{j \frown \min L_r\mid j \in I_r\} \cup \bigcup_{r=1}^{s^\prime-1} \{\max I_r \smile m\mid m \in L_{r+1} \} \]
where $I_r, L_r \subseteq [n]$ with $\{i\} = I_1 < \cdots < I_s \leq L_{s^\prime} < \cdots < L_1 = \{l\}$.
\end{defn}

\noindent For subsets $I, L \subseteq [n]$ we say $I < L$ if $i <l$ for each $i \in I$ and $l \in L$.  If $\max I = \min L$, we say $I \leq L$.  It follows that a generalized shell with $|I_r|=1$ and $|L_r|=1$ for all $r$ is simply a shell in the sense of Definition \ref{simpleshell}.  Some generalized shells of size $6$ and width $6-2$ are
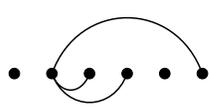
\[ \begin{tikzpicture}
 \draw[fill=black] (-1,0) circle (.07);
 \draw[fill=black] (-.5,0) circle (.07);
\draw[fill=black] (.0,0) circle (.07);
\draw[fill=black] (.5,0) circle (.07);
\draw[fill=black] (1,0) circle (.07);
\draw[fill=black] (1.5,0) circle (.07);
\draw[line width=.2mm] (-.5,0) to[in=-110, out=-70, circle, looseness=1.6] (0,0);
\draw[line width=.2mm] (-.5,0) to[in=110, out=70, circle, looseness=1.35] (1.5,0);
\draw[line width=.2mm] (-.5,0) to[in=-110, out=-70, circle, looseness=1.4] (.5,0);
\end{tikzpicture}\qquad\qquad\qquad\qquad\qquad\begin{tikzpicture}
 \draw[fill=black] (-1,0) circle (.07);
 \draw[fill=black] (-.5,0) circle (.07);
\draw[fill=black] (.0,0) circle (.07);
\draw[fill=black] (.5,0) circle (.07);
\draw[fill=black] (1,0) circle (.07);
\draw[fill=black] (1.5,0) circle (.07);
\draw[line width=.2mm] (.5,0) to[in=110, out=70, circle, looseness=1.6] (1,0);
\draw[line width=.2mm] (0,0) to[in=110, out=70, circle, looseness=1.4] (1,0);
\draw[line width=.2mm] (-.5,0) to[in=110, out=70, circle, looseness=1.35] (1.5,0);
\draw[line width=.2mm] (-.5,0) to[in=-110, out=-70, circle, looseness=1.4] (1,0);
\end{tikzpicture}\]\vspace{-.3in}
\[\qquad\quad\{2 \frown 6\} \cup \{2\smile m \mid m \in \{3, 4\} \}, \qquad\qquad \{2 \frown 6\}\cup \{j \frown 5 \mid j \in \{3,4\} \} \cup \{ 2 \smile 5 \}.\]

A \textit{labeled shell} is a pair $(\varsigma, \tau)$ for a generalized shell $\varsigma$ and a map $\tau: \varsigma \to \mathbb{Z}_{\geq 0}$.  We say the labeling $\tau$ is \textit{strict} if every pair of arcs $(i,l), (j,m) \in \varsigma$ with $\dim(i,j) > \dim(j,m)$ satisfies $\tau(i,j) < \tau(j,m)$, and $\tau(j,m) \neq \tau (i,l)+1$ if $i=j$ or $l=m$.  If $\tau(i,j) = a$, we write the labeled arc as $(i,j;a)$.  When the orientation of the arc is specified we write $(i \frown j;a)$ or $(i \smile j;a)$.  For example, in the case of the shell 

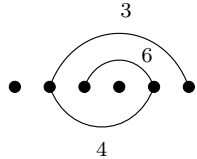
\[
\begin{tikzpicture}[scale=1.85]
\draw[fill=black] (1.75,0) circle (.04);
\draw[fill=black] (2,0) circle (.04);
\draw[fill=black] (2.25,0) circle (.04);
\draw[fill=black] (2.5,0) circle (.04);
\draw[fill=black] (2.75,0) circle (.04);
\draw[fill=black] (3,0) circle (.04);
\draw[line width=.15mm] (2,0) to[in=110, out=70, circle, looseness= 1.4] (3,0);
\draw[line width=.15mm] (2,0) to[in=-110, out=-70, circle, looseness= 1.4] (2.75,0);
\draw[line width=.15mm] (2.25,0) to[in=110, out=70, circle, looseness= 1.4] (2.75,0);
\node at (2.55,.55) {\scriptsize{3}};
\node at (2.375,-.45) {\scriptsize{4}};
\node at (2.7,.23) {\scriptsize{6}};
\end{tikzpicture}\]\vspace{-.3in}
\[ \{(2\frown 6; 3), (3\frown 5; 6), (2 \smile 5; 4)\}.\]
From strictly labeled shells, we define the key notion shell tableaux.

 \begin{defn}\label{tableaux}
A \textit{shell tableau} $T = (\varsigma^1, \ldots, \varsigma^k)$ of \textit{length} $k$ is a sequence of strictly labeled shells $\varsigma^r$ of size $n$ and width $n-i_r$ such that 
\begin{enumerate}
\item for $1 \leq r < k, \varsigma^r = \{ (n \frown n ; a)\}$ or $|\varsigma^r| \geq 2$, and $\varsigma^k = \{ (i_k \frown n; a)\}$;
\item each arc has a distinct label in $\{1, 2, \ldots, \sum_{r=1}^k |\varsigma^r| \}$;
\item the two smallest labels of each labeled shell $\varsigma^r$ are less than the smallest label in $\varsigma^{r+1}$;
\item for $l \neq m$ and $i<j \leq \min \{l,m\}$, if $(i,l;a) \in \varsigma^{r_l} $ then there exists a minimal $b>a$ such that $(i,m;b) \in \varsigma^{r_m}$ if and only if $(j,\min\{l,m\};b+1)\in \varsigma^{r_{\min\{l,m\}}}$;
\item for $i \neq j$ and $\max \{i,j\}\leq l <m$, if $(j,m;a)\in \varsigma^{r_j}$ then there exists a minimal $b>a$ such that $(i,m;b)\in \varsigma^{r_i}$ if and only if $(\max\{i,j\},l;b+1)\in \varsigma^{r_{\max\{i,j\}}}$.
\end{enumerate}

\end{defn}
\noindent  Conditions 1--3 provide the basic set up of the shells and labeling that are analogous to the condition of increasing entires along the rows and columns in standard Young tableaux.  Intuitively conditions $4$ and $5$ say a strictly labeled shell in a shell tableau has inner whorls if and only if its outer whorl conflicts with the outer whorl of another shell.  As an example consider the tableau \[
\begin{tikzpicture}
\node at (-.75,0) {$T$= \Bigg(};
  \draw[fill=black] (0,0) circle (.04);
\draw[fill=black] (.25,0) circle (.04);
\draw[fill=black] (.5,0) circle (.04);
\draw[fill=black] (.75,0) circle (.04);
\draw[fill=black] (1,0) circle (.04);
\draw[fill=black] (1.25,0) circle (.04);
\draw[line width=.15mm] (0,0) to[in=110, out=70, circle, looseness=1.3] (1.25,0);
\draw[line width=.15mm] (0,0) to[in=-110, out=-70, circle, looseness=1.4] (.75,0);
\node at (.625,.6) {\tiny{1}};
\node at (.375,-.45) {\tiny{2}};
\node at (1.5,0) {,};
\draw[fill=black] (1.75,0) circle (.04);
\draw[fill=black] (2,0) circle (.04);
\draw[fill=black] (2.25,0) circle (.04);
\draw[fill=black] (2.5,0) circle (.04);
\draw[fill=black] (2.75,0) circle (.04);
\draw[fill=black] (3,0) circle (.04);
\draw[line width=.15mm] (2,0) to[in=110, out=70, circle, looseness= 1.4] (3,0);
\draw[line width=.15mm] (2,0) to[in=-110, out=-70, circle, looseness= 1.4] (2.75,0);
\draw[line width=.15mm] (2.25,0) to[in=110, out=70, circle, looseness= 1.4] (2.75,0);
\node at (2.55,.55) {\tiny{3}};
\node at (2.375,-.45) {\tiny{4}};
\node at (2.7,.23) {\tiny{6}};
\node at (3.25,-.1) {,};
\draw[fill=black] (3.5,0) circle (.04);
\draw[fill=black] (3.75,0) circle (.04);
\draw[fill=black] (4,0) circle (.04);
\draw[fill=black] (4.25,0) circle (.04);
\draw[fill=black] (4.5,0) circle (.04);
\draw[fill=black] (4.75,0) circle (.04);
\draw[line width= .15mm] (3.75,0) to[in=110, out=70, circle, looseness= 1.4] (4.75,0);
\draw[line width= .15mm] (3.75,0) to[in=-110, out=-70, circle, looseness= 1.7] (4,0);
\node at (4.25, .55) {\tiny{5}};
\node at (3.875, -.25) {\tiny{7}};
\node at (5,-.1) {,};
\draw[fill=black] (5.25,0) circle (.04);
\draw[fill=black] (5.5,0) circle (.04);
\draw[fill=black] (5.75,0) circle (.04);
\draw[fill=black] (6,0) circle (.04);
\draw[fill=black] (6.25,0) circle (.04);
\draw[fill=black] (6.5,0) circle (.04);
\draw[line width=.15mm] (6.5,0) to[in=50, out=130, loop] (6.5,0);
\node at (6.5, .4) {\tiny{8}};
\node at (7,0) {\Bigg)};
\end{tikzpicture}\]
of length $4$.  By condition 4, the inner half whorl $(3 \frown 5;6)$ lies in $\varsigma^2$ since $(2\smile 5; 4) \in \varsigma^2$ conflicts with $(2 \frown 6;5) \in \varsigma^3$.

Let $\mathcal{ST}_k$ denote the set of shell tableaux of length $k$.
\begin{defn}Define the map
\begin{align*}
  \sh: (\mathbb{Z}_{\geq 0}, \mathcal{ST}_k)& \longrightarrow \text{Set of Arcs}  \\
  (a,T) &\longmapsto \bigcup_{r=1}^k \left\{ (i,l) \left| \begin{array}{c} i \neq l \ \text{and}\ \tau((i,l)) \ \text{is maximal}\\ \text{among all labels} \ b\in\varsigma^r \ \text{with}\ b\leq a \end{array} \right\}, \right.
\end{align*}
and $\sh(T) = \sh(|T|, T)$ be the \textit{shape} of a shell tableau $T$.
\end{defn}
\noindent For $T$ in the example above, we have
\[\sh(T) = \begin{tikzpicture}[scale=1.9]
\draw[fill=black] (9.75,0) circle (.04);
\draw[fill=black] (10,0) circle (.04);
\draw[fill=black] (10.25,0) circle (.04);
\draw[fill=black] (10.5,0) circle (.04);
\draw[fill=black] (10.75,0) circle (.04);
\draw[fill=black] (11,0) circle (.04);
\draw[line width=.15mm] (9.75,0) to[in=110, out=70, circle, looseness=1.4] (10.5,0);
\draw[line width=.15mm] (10,0) to[in=110, out=70, circle, looseness=1.7] (10.25,0);
\draw[line width=.15mm] (10.25,0) to[in=110, out=70, circle, looseness=1.4] (10.75,0);
\end{tikzpicture}\quad\]
because $\tau(1,4) = 2$ is the maximal label $\varsigma^1$, $\tau(3, 5) = 6$ is the maximal label in $\varsigma^2$ and $\tau(2, 3)$ is the maximal label in $\varsigma^3$.  For $\lambda \vdash [n]$, let $\mathcal{ST}_k(\lambda)$ denote the set of shell tableaux of shape $\lambda.$

\begin{thm}\label{tableauxbijection} Let $\lambda \in \hat{Z}_k$.  There is a bijection between $\mathcal{P}_k(\lambda)$ and $\mathcal{ST}_k(\lambda)$.
\end{thm}

\begin{proof}
Let $\lambda \in \hat{Z}_k$.  Given a path $P = (\lambda^0, \lambda^{\frac{1}{2}}, \lambda^1, \ldots, \lambda^{k-\frac{1}{2}}, \lambda^k)\in \mathcal{P}_k(\lambda)$ we will recursively define a sequence 
\[T_0, T_{\frac{1}{2}}, T_1, \ldots, T_{k-\frac{1}{2}}, T_k\]
 where $T_j$ is a shell tableau of length $j$ and shape $\lambda^j$, and $T_{j+\frac{1}{2}}$ is a shell tableau of length $j+1$ and shape $\lambda^{j+\frac{1}{2}}$.  Let $T_0$ be the empty shell tableau of length $0$.
 
 \begin{enumerate}
 
\item If $\lambda^{j+\frac{1}{2}} = \lambda^j$, we define
\[ T_{j+\frac{1}{2}} = (\varsigma_{j+\frac{1}{2}}^{1}, \varsigma_{j+\frac{1}{2}}^{2}, \ldots, \varsigma_{j+\frac{1}{2}}^{j+1}),\] 
 where
\[ \varsigma_{j+\frac{1}{2} }^{r} =\left\{\begin{array}{ll}
\varsigma_{j}^{r} &  \text{if}\ r< j+1, \\
(n, n; |T_j|+1) & \text{if} \ r = j+1. \end{array}\right.
\]

\item If $\lambda^{j+\frac{1}{2}} \neq \lambda^j$, suppose 
\[ \bigcup_{s=1}^t \{i_s \frown l_s\} \cup \bigcup_{s=1}^{t} \{i_s \smile l_{s+1}\} \ \]
is the shell created by the symmetric difference of $\lambda^{r}$ and $\lambda^{r+\frac{1}{2}}$.  Since $\sh(T_j) = \lambda^j$, for $1 \leq s \leq t$, $(i_s, l_s; a_s)$ is an arc with maximal label $a_s \leq |T_j|$ in a diagram $\varsigma_{j}^{r_s}$ of $T_j$. Let
\[ T_{j +\frac{1}{2}} = (\varsigma_{j+\frac{1}{2}}^{1}, \varsigma_{j+\frac{1}{2}}^{2}, \ldots, \varsigma_{j+\frac{1}{2}}^{j+1}),\]
where
\[ \varsigma_{j+\frac{1}{2}}^{r} =\left\{\begin{array}{ll}
\varsigma_j^r &  \text{if}\ r \neq r_s\ \text{for any}\ s, \\
\varsigma_j^r + (i_s, l_{s+1}; |T_j|+s) & \text{if}\ r = r_s\ \text{for some}\ s ,\\
(n,n; |T_j| + t+1) & \text{if}\ r = j+1.
\end{array}\right.
\]

\item If $\lambda^{j+1} = \lambda^{j+\frac{1}{2}}$, define $T_{j+1} = T_{j+\frac{1}{2}}$.

\item If $\lambda^{j+1} \neq \lambda^{j+\frac{1}{2}}$, suppose 
\[ \bigcup_{s=1}^t \{i_s \frown l_s\} \cup \bigcup_{s=1}^{t-1} \{i_s \smile l_{s+1}\} \ \]
is the shell created by the symmetric difference of $\lambda^{r}$ and $\lambda^{r+\frac{1}{2}}$.  Since $\sh(T_{j+\frac{1}{2}}) = \lambda^{j+\frac{1}{2}} $ then for $1 \leq s \leq t-1$, $(i_s, l_{s+1}; a_s)$ is an arc with maximal label $a_s \leq |T_{j+\frac{1}{2}}|$ in a distinct diagram $\varsigma_{j+\frac{1}{2}}^{r_s}$ in $T_{j+\frac{1}{2}}$.   We define
\[ T_{j+1} = (\varsigma_{j+1}^{1}, \varsigma_{j+1}^{2}, \ldots, \varsigma_{j+1}^{j+1}),\] 
where
\[ \varsigma_{j+1}^{r} =\left\{\begin{array}{ll}
\varsigma_{j+\frac{1}{2}}^{r} &  \text{if}\ r \neq r_s\ \text{for any}\ s, \\
\varsigma_{j+\frac{1}{2}}^{r} + (i_{s+1}, l_{s+1}; |T_{j+\frac{1}{2}}|+s) & \text{if}\ r = r_s\ \text{for some}\ s, \\ 
(i_1, l_1, |T_{j+\frac{1}{2}}|) & \text{if} \ r = j+1. \end{array}\right.
\]

\end{enumerate}

In the above construction, we have $T_{\frac{1}{2}} = (\varsigma_{\frac{1}{2}}^1)$ where $\varsigma_{\frac{1}{2}}^1 = \{ (n\frown n; 1)\}$ is a shell tableau of length 1 and shape $\emptyset$.  If $T_j$ is a shell tableau of length $j$ and shape $\lambda^j$, then $T_{j+\frac{1}{2}}$ has length $j+1$ and
\[ \sh(T_{j+\frac{1}{2}} )=  (\lambda^j \cap \lambda^{j+\frac{1}{2}}) \cup (\lambda^{j+\frac{1}{2}}- \lambda^j) = \lambda^{j+\frac{1}{2}}.\]
It is straightforward to check that $T_{j+\frac{1}{2}}$ satisfies conditions $1$--$4$.  Since $T_j$ is a shell tableau, it suffices to prove condition 5 for the arcs $(i_s, l_{s+1}; |T_j|+s)$.  For $s>1$, consider $(i_s, l_s; a_s) \in \varsigma_{j+\frac{1}{2}}^{r_s}$. Then $(i_{s-1}, l_s; |T_j|+s-1)$ lies in $\varsigma_{j+\frac{1}{2}}^{r_{s-1}}$ where $|T_j|+s-1> a_s$ is minimal, and $(i_s, l_{s+1}; |T_{j}|+s) \in \varsigma_{j+\frac{1}{2}}^{r_s}$.  Thus condition 5 holds, so $T_{j+\frac{1}{2}}$ is in fact a shell tableau.  A similar argument can be used to verify each $T_j$ is a shell tableau of length $j$ and shape $\lambda^j$.

For $\lambda \in \hat{Z}_k$, define
\[\begin{array}{ccc}\varphi: \mathcal{P}_k(\lambda) &\longrightarrow& \mathcal{ST}_k(\lambda) \\ P & \mapsto& T_k. \end{array} \]

The map $\varphi$ is bijective since the construction of the sequence of shell tableaux can be reversed as follows.  Given a shell tableau $T = (\varsigma^1, \varsigma^2, \ldots, \varsigma^k)$ of shape $\lambda$, let $T_k= T$. 
 \begin{enumerate}
 
\item If $\varsigma^j = \{(n \frown n;a)\}$, define $T_{j-\frac{1}{2}} = T_{j}$.

\item If $(i \frown n;a) \in \varsigma^j$ for $i<n$, let $ (i_1, l_1; a_1), (i_2, l_2 ; a_2), \ldots, (i_t, l_t; a_t)$
be the arcs in $T_j$ with $i_s \frown l_s \in \varsigma_j^{r_s}$ and $a_s \geq a$.  We define
\[ T_{j-\frac{1}{2}} = (\varsigma_{j-\frac{1}{2}}^{1}, \varsigma_{j-\frac{1}{2}}^{2}, \ldots, \varsigma_{j-\frac{1}{2}}^{j}),\] 
where
\[ \varsigma_{j-\frac{1}{2}}^{r} =\left\{\begin{array}{ll}
\varsigma_{j}^{r} &  \text{if}\ r \neq r_s\ \text{for any}\ s ,\\
\varsigma_{j}^{r} - (i_{s}, l_{s}; a_s) & \text{if}\ r = r_s\ \text{for some}\ s ,\\ 
(n, n, |T_{j}|-t+1) & \text{if} \ r = j. \end{array}\right.
 \]
 
\item If $\varsigma^j = \{(n \frown n;a)\}$, define $T_{j} = (\varsigma_{j}^{1}, \varsigma_{j}^{2}, \ldots, \varsigma_{j}^{j-1})$ where $\varsigma_{j }^{r} =\varsigma_{j+\frac{1}{2}}^{r}$.

\item If $(i \frown n;a) \in \varsigma^j$ for $i<n$, let $ (i_1, l_1; a_1), (i_2, l_2 ; a_2), \ldots, (i_t, l_t; a_t)$
be the arcs in $T_j$ with $i_s \frown l_s \in \varsigma_j^{r_s}$ and $a_s > a$.  We define
\[ T_{j} = (\varsigma_{j}^{1}, \varsigma_{j}^{2}, \ldots, \varsigma_{j}^{j-1}),\]
where
\[ \varsigma_{j}^{r} =\left\{\begin{array}{ll}
\varsigma_{j+\frac{1}{2}}^r &  \text{if}\ r \neq r_s\ \text{for any}\ s, \\
\varsigma_{j+\frac{1}{2}}^r - (i_s, l_{s}; a_s) & \text{if}\ r = r_s\ \text{for some}\ s .\\
\end{array}\right.
\]
\end{enumerate}

Therefore the inverse of $\varphi$ is
\[\begin{array}{ccc}\varphi^{-1}: \mathcal{ST}_k(\lambda) &\longrightarrow& \mathcal{P}_k(\lambda)\\ T & \mapsto& P=(\sh(T_1), \sh(T_2), \ldots, \sh(T_k)). \end{array} \]

\end{proof}

\begin{example}\label{extableaux}
For the path 
\[
\begin{tikzpicture}
\node at (-1.1,0) {$P = \Bigg(\emptyset, \emptyset,$};
  \draw[fill=black] (0,0) circle (.04);
\draw[fill=black] (.25,0) circle (.04);
\draw[fill=black] (.5,0) circle (.04);
\draw[fill=black] (.75,0) circle (.04);
\draw[fill=black] (1,0) circle (.04);
\draw[fill=black] (1.25,0) circle (.04);
\draw[line width=.15mm] (0,0) to[in=110, out=70, circle, looseness=1.3] (1.25,0);
\node at (1.5,-.1) {,};
  \draw[fill=black] (1.75,0) circle (.04);
\draw[fill=black] (2,0) circle (.04);
\draw[fill=black] (2.25,0) circle (.04);
\draw[fill=black] (2.5,0) circle (.04);
\draw[fill=black] (2.75,0) circle (.04);
\draw[line width=.15mm] (1.75,0) to[in=110, out=70, circle, looseness=1.4] (2.5,0);
\node at (3,-.1) {,};
  \draw[fill=black] (3.25,0) circle (.04);
\draw[fill=black] (3.5,0) circle (.04);
\draw[fill=black] (3.75,0) circle (.04);
\draw[fill=black] (4,0) circle (.04);
\draw[fill=black] (4.25,0) circle (.04);
\draw[fill=black] (4.5,0) circle (.04);
\draw[line width=.15mm] (3.25,0) to[in=110, out=70, circle, looseness=1.4] (4,0);
\draw[line width=.15mm] (3.5,0) to[in=110, out=70, circle, looseness=1.4] (4.5,0);
\node at (4.75,-.1) {,};
  \draw[fill=black] (5,0) circle (.04);
\draw[fill=black] (5.25,0) circle (.04);
\draw[fill=black] (5.5,0) circle (.04);
\draw[fill=black] (5.75,0) circle (.04);
\draw[fill=black] (6,0) circle (.04);
\draw[line width=.15mm] (5,0) to[in=110, out=70, circle, looseness=1.4] (5.75,0);
\draw[line width=.15mm] (5.25,0) to[in=110, out=70, circle, looseness=1.4] (6,0);
\node at (6.25,-.1) {,};
\draw[fill=black] (6.5,0) circle (.04);
\draw[fill=black] (6.75,0) circle (.04);
\draw[fill=black] (7,0) circle (.04);
\draw[fill=black] (7.25,0) circle (.04);
\draw[fill=black] (7.5,0) circle (.04);
\draw[fill=black] (7.75,0) circle (.04);
\draw[line width=.15mm] (6.5,0) to[in=110, out=70, circle, looseness=1.4] (7.25,0);
\draw[line width=.15mm] (6.75,0) to[in=110, out=70, circle, looseness=1.4] (7.75,0);
\draw[line width=.15mm] (7,0) to[in=110, out=70, circle, looseness=1.4] (7.5,0);
\node at (8,-.1) {,};
\draw[fill=black] (8.25,0) circle (.04);
\draw[fill=black] (8.5,0) circle (.04);
\draw[fill=black] (8.75,0) circle (.04);
\draw[fill=black] (9,0) circle (.04);
\draw[fill=black] (9.25,0) circle (.04);
\draw[line width=.15mm] (8.25,0) to[in=110, out=70, circle, looseness=1.4] (9,0);
\draw[line width=.15mm] (8.5,0) to[in=110, out=70, circle, looseness=1.7] (8.75,0);
\draw[line width=.15mm] (8.75,0) to[in=110, out=70, circle, looseness=1.4] (9.25,0);
\node at (9.5,-.1) {,};
\draw[fill=black] (9.75,0) circle (.04);
\draw[fill=black] (10,0) circle (.04);
\draw[fill=black] (10.25,0) circle (.04);
\draw[fill=black] (10.5,0) circle (.04);
\draw[fill=black] (10.75,0) circle (.04);
\draw[fill=black] (11,0) circle (.04);
\draw[line width=.15mm] (9.75,0) to[in=110, out=70, circle, looseness=1.4] (10.5,0);
\draw[line width=.15mm] (10,0) to[in=110, out=70, circle, looseness=1.7] (10.25,0);
\draw[line width=.15mm] (10.25,0) to[in=110, out=70, circle, looseness=1.4] (10.75,0);
\node at (11.25,0) {\Bigg)};
\end{tikzpicture}\]
the sequence of shell tableaux is
\begin{eqnarray*}
&&\ T_0 \ \, =\ (\, )\\
&& \begin{tikzpicture}
\node at (-.75,0) {$T_{\frac{1}{2}\ } = \Bigg($};
  \draw[fill=black] (0,0) circle (.04);
\draw[fill=black] (.25,0) circle (.04);
\draw[fill=black] (.5,0) circle (.04);
\draw[fill=black] (.75,0) circle (.04);
\draw[fill=black] (1,0) circle (.04);
\draw[fill=black] (1.25,0) circle (.04);
\draw[line width=.15mm] (1.25,0) to[in=50, out=130, loop] (1.25,0);
\node at (1.25,.4) {\tiny{1}};
\node at (1.5,0) {$\Bigg)$};
\end{tikzpicture}\\
&& \begin{tikzpicture}
\node at (-.75,0) {$T_{1\ \, } = \Bigg($};
  \draw[fill=black] (0,0) circle (.04);
\draw[fill=black] (.25,0) circle (.04);
\draw[fill=black] (.5,0) circle (.04);
\draw[fill=black] (.75,0) circle (.04);
\draw[fill=black] (1,0) circle (.04);
\draw[fill=black] (1.25,0) circle (.04);
\draw[line width=.15mm] (0,0) to[in=110, out=70, circle, looseness=1.3] (1.25,0);
\node at (.625,.6) {\tiny{1}};
\node at (1.5,0) {$\Bigg)$};
\end{tikzpicture}\\
&& \begin{tikzpicture}
\node at (-.75,0) {$T_{1\frac{1}{2}} = \Bigg($};
  \draw[fill=black] (0,0) circle (.04);
\draw[fill=black] (.25,0) circle (.04);
\draw[fill=black] (.5,0) circle (.04);
\draw[fill=black] (.75,0) circle (.04);
\draw[fill=black] (1,0) circle (.04);
\draw[fill=black] (1.25,0) circle (.04);
\draw[line width=.15mm] (0,0) to[in=110, out=70, circle, looseness=1.3] (1.25,0);
\draw[line width=.15mm] (0,0) to[in=-110, out=-70, circle, looseness=1.4] (.75,0);
\node at (.625,.6) {\tiny{1}};
\node at (.375,-.45) {\tiny{2}};
\node at (1.5,0) {,};
\draw[fill=black] (1.75,0) circle (.04);
\draw[fill=black] (2,0) circle (.04);
\draw[fill=black] (2.25,0) circle (.04);
\draw[fill=black] (2.5,0) circle (.04);
\draw[fill=black] (2.75,0) circle (.04);
\draw[fill=black] (3,0) circle (.04);
\draw[line width=.15mm] (3,0) to[in=50, out=130, loop] (3,0);
\node at (3,.4) {\tiny{3}};
\node at (3.25,0) {$\Bigg)$};
\end{tikzpicture}\\
&& \begin{tikzpicture}
\node at (-.75,0) {$T_{2\ \, } = \Bigg($};
  \draw[fill=black] (0,0) circle (.04);
\draw[fill=black] (.25,0) circle (.04);
\draw[fill=black] (.5,0) circle (.04);
\draw[fill=black] (.75,0) circle (.04);
\draw[fill=black] (1,0) circle (.04);
\draw[fill=black] (1.25,0) circle (.04);
\draw[line width=.15mm] (0,0) to[in=110, out=70, circle, looseness=1.3] (1.25,0);
\draw[line width=.15mm] (0,0) to[in=-110, out=-70, circle, looseness=1.4] (.75,0);
\node at (.625,.6) {\tiny{1}};
\node at (.375,-.45) {\tiny{2}};
\node at (1.5,0) {,};
\draw[fill=black] (1.75,0) circle (.04);
\draw[fill=black] (2,0) circle (.04);
\draw[fill=black] (2.25,0) circle (.04);
\draw[fill=black] (2.5,0) circle (.04);
\draw[fill=black] (2.75,0) circle (.04);
\draw[fill=black] (3,0) circle (.04);
\draw[line width=.15mm] (2,0) to[in=110, out=70, circle, looseness= 1.4] (3,0);
\node at (2.5,.55) {\tiny{3}};
\node at (3.25,0) {$\Bigg)$};
\end{tikzpicture}\\
&& \begin{tikzpicture}
\node at (-.75,0) {$T_{2\frac{1}{2}} = \Bigg($};
  \draw[fill=black] (0,0) circle (.04);
\draw[fill=black] (.25,0) circle (.04);
\draw[fill=black] (.5,0) circle (.04);
\draw[fill=black] (.75,0) circle (.04);
\draw[fill=black] (1,0) circle (.04);
\draw[fill=black] (1.25,0) circle (.04);
\draw[line width=.15mm] (0,0) to[in=110, out=70, circle, looseness=1.3] (1.25,0);
\draw[line width=.15mm] (0,0) to[in=-110, out=-70, circle, looseness=1.4] (.75,0);
\node at (.625,.6) {\tiny{1}};
\node at (.375,-.45) {\tiny{2}};
\node at (1.5,0) {,};
\draw[fill=black] (1.75,0) circle (.04);
\draw[fill=black] (2,0) circle (.04);
\draw[fill=black] (2.25,0) circle (.04);
\draw[fill=black] (2.5,0) circle (.04);
\draw[fill=black] (2.75,0) circle (.04);
\draw[fill=black] (3,0) circle (.04);
\draw[line width=.15mm] (2,0) to[in=110, out=70, circle, looseness= 1.4] (3,0);
\draw[line width=.15mm] (2,0) to[in=-110, out=-70, circle, looseness= 1.4] (2.75,0);
\node at (2.55,.55) {\tiny{3}};
\node at (2.375,-.45) {\tiny{4}};
\node at (3.25,-.1) {,};
\draw[fill=black] (3.5,0) circle (.04);
\draw[fill=black] (3.75,0) circle (.04);
\draw[fill=black] (4,0) circle (.04);
\draw[fill=black] (4.25,0) circle (.04);
\draw[fill=black] (4.5,0) circle (.04);
\draw[fill=black] (4.75,0) circle (.04);
\draw[line width= .15mm] (4.75,0) to[in=50, out=130, loop] (4.75);
\node at (4.75, .4) {\tiny{5}};
\node at (5,0) {$\Bigg)$};
\end{tikzpicture}\\
&& \begin{tikzpicture}
\node at (-.75,0) {$T_{3 \ \, } = \Bigg($};
  \draw[fill=black] (0,0) circle (.04);
\draw[fill=black] (.25,0) circle (.04);
\draw[fill=black] (.5,0) circle (.04);
\draw[fill=black] (.75,0) circle (.04);
\draw[fill=black] (1,0) circle (.04);
\draw[fill=black] (1.25,0) circle (.04);
\draw[line width=.15mm] (0,0) to[in=110, out=70, circle, looseness=1.3] (1.25,0);
\draw[line width=.15mm] (0,0) to[in=-110, out=-70, circle, looseness=1.4] (.75,0);
\node at (.625,.6) {\tiny{1}};
\node at (.375,-.45) {\tiny{2}};
\node at (1.5,0) {,};
\draw[fill=black] (1.75,0) circle (.04);
\draw[fill=black] (2,0) circle (.04);
\draw[fill=black] (2.25,0) circle (.04);
\draw[fill=black] (2.5,0) circle (.04);
\draw[fill=black] (2.75,0) circle (.04);
\draw[fill=black] (3,0) circle (.04);
\draw[line width=.15mm] (2,0) to[in=110, out=70, circle, looseness= 1.4] (3,0);
\draw[line width=.15mm] (2,0) to[in=-110, out=-70, circle, looseness= 1.4] (2.75,0);
\draw[line width=.15mm] (2.25,0) to[in=110, out=70, circle, looseness= 1.4] (2.75,0);
\node at (2.55,.55) {\tiny{3}};
\node at (2.375,-.45) {\tiny{4}};
\node at (2.7,.23) {\tiny{6}};
\node at (3.25,-.1) {,};
\draw[fill=black] (3.5,0) circle (.04);
\draw[fill=black] (3.75,0) circle (.04);
\draw[fill=black] (4,0) circle (.04);
\draw[fill=black] (4.25,0) circle (.04);
\draw[fill=black] (4.5,0) circle (.04);
\draw[fill=black] (4.75,0) circle (.04);
\draw[line width= .15mm] (3.75,0) to[in=110, out=70, circle, looseness= 1.4] (4.75,0);
\node at (4.25, .55) {\tiny{5}};
\node at (5,0) {$\Bigg)$};
\end{tikzpicture}\\
&& \begin{tikzpicture}
\node at (-.75,0) {$T_{3\frac{1}{2}} = \Bigg($};
  \draw[fill=black] (0,0) circle (.04);
\draw[fill=black] (.25,0) circle (.04);
\draw[fill=black] (.5,0) circle (.04);
\draw[fill=black] (.75,0) circle (.04);
\draw[fill=black] (1,0) circle (.04);
\draw[fill=black] (1.25,0) circle (.04);
\draw[line width=.15mm] (0,0) to[in=110, out=70, circle, looseness=1.3] (1.25,0);
\draw[line width=.15mm] (0,0) to[in=-110, out=-70, circle, looseness=1.4] (.75,0);
\node at (.625,.6) {\tiny{1}};
\node at (.375,-.45) {\tiny{2}};
\node at (1.5,0) {,};
\draw[fill=black] (1.75,0) circle (.04);
\draw[fill=black] (2,0) circle (.04);
\draw[fill=black] (2.25,0) circle (.04);
\draw[fill=black] (2.5,0) circle (.04);
\draw[fill=black] (2.75,0) circle (.04);
\draw[fill=black] (3,0) circle (.04);
\draw[line width=.15mm] (2,0) to[in=110, out=70, circle, looseness= 1.4] (3,0);
\draw[line width=.15mm] (2,0) to[in=-110, out=-70, circle, looseness= 1.4] (2.75,0);
\draw[line width=.15mm] (2.25,0) to[in=110, out=70, circle, looseness= 1.4] (2.75,0);
\node at (2.55,.55) {\tiny{3}};
\node at (2.375,-.45) {\tiny{4}};
\node at (2.7,.23) {\tiny{6}};
\node at (3.25,-.1) {,};
\draw[fill=black] (3.5,0) circle (.04);
\draw[fill=black] (3.75,0) circle (.04);
\draw[fill=black] (4,0) circle (.04);
\draw[fill=black] (4.25,0) circle (.04);
\draw[fill=black] (4.5,0) circle (.04);
\draw[fill=black] (4.75,0) circle (.04);
\draw[line width= .15mm] (3.75,0) to[in=110, out=70, circle, looseness= 1.4] (4.75,0);
\draw[line width= .15mm] (3.75,0) to[in=-110, out=-70, circle, looseness= 1.7] (4,0);
\node at (4.25, .55) {\tiny{5}};
\node at (3.875, -.25) {\tiny{7}};
\node at (5,-.1) {,};
\draw[fill=black] (5.25,0) circle (.04);
\draw[fill=black] (5.5,0) circle (.04);
\draw[fill=black] (5.75,0) circle (.04);
\draw[fill=black] (6,0) circle (.04);
\draw[fill=black] (6.25,0) circle (.04);
\draw[fill=black] (6.5,0) circle (.04);
\draw[line width=.15mm] (6.5,0) to[in=50, out=130, loop] (6.5,0);
\node at (6.5, .4) {\tiny{8}};
\node at (6.75,0) {$\Bigg)$};
\end{tikzpicture}\\
&& \begin{tikzpicture}
\node at (-.75,0) {$T_{4 \ \ } = \Bigg($};
  \draw[fill=black] (0,0) circle (.04);
\draw[fill=black] (.25,0) circle (.04);
\draw[fill=black] (.5,0) circle (.04);
\draw[fill=black] (.75,0) circle (.04);
\draw[fill=black] (1,0) circle (.04);
\draw[fill=black] (1.25,0) circle (.04);
\draw[line width=.15mm] (0,0) to[in=110, out=70, circle, looseness=1.3] (1.25,0);
\draw[line width=.15mm] (0,0) to[in=-110, out=-70, circle, looseness=1.4] (.75,0);
\node at (.625,.6) {\tiny{1}};
\node at (.375,-.45) {\tiny{2}};
\node at (1.5,0) {,};
\draw[fill=black] (1.75,0) circle (.04);
\draw[fill=black] (2,0) circle (.04);
\draw[fill=black] (2.25,0) circle (.04);
\draw[fill=black] (2.5,0) circle (.04);
\draw[fill=black] (2.75,0) circle (.04);
\draw[fill=black] (3,0) circle (.04);
\draw[line width=.15mm] (2,0) to[in=110, out=70, circle, looseness= 1.4] (3,0);
\draw[line width=.15mm] (2,0) to[in=-110, out=-70, circle, looseness= 1.4] (2.75,0);
\draw[line width=.15mm] (2.25,0) to[in=110, out=70, circle, looseness= 1.4] (2.75,0);
\node at (2.55,.55) {\tiny{3}};
\node at (2.375,-.45) {\tiny{4}};
\node at (2.7,.23) {\tiny{6}};
\node at (3.25,-.1) {,};
\draw[fill=black] (3.5,0) circle (.04);
\draw[fill=black] (3.75,0) circle (.04);
\draw[fill=black] (4,0) circle (.04);
\draw[fill=black] (4.25,0) circle (.04);
\draw[fill=black] (4.5,0) circle (.04);
\draw[fill=black] (4.75,0) circle (.04);
\draw[line width= .15mm] (3.75,0) to[in=110, out=70, circle, looseness= 1.4] (4.75,0);
\draw[line width= .15mm] (3.75,0) to[in=-110, out=-70, circle, looseness= 1.7] (4,0);
\node at (4.25, .55) {\tiny{5}};
\node at (3.875, -.25) {\tiny{7}};
\node at (5,-.1) {,};
\draw[fill=black] (5.25,0) circle (.04);
\draw[fill=black] (5.5,0) circle (.04);
\draw[fill=black] (5.75,0) circle (.04);
\draw[fill=black] (6,0) circle (.04);
\draw[fill=black] (6.25,0) circle (.04);
\draw[fill=black] (6.5,0) circle (.04);
\draw[line width=.15mm] (6.5,0) to[in=50, out=130, loop] (6.5,0);
\node at (6.5, .4) {\tiny{8}};
\node at (6.75,0) {$\Bigg).$};
\end{tikzpicture}
 \end{eqnarray*}

\end{example}
\noindent Note that each shell $\varsigma^r$ keeps track of the arc introduced at the $r$th row of the Bratteli diagram from inducing $\Res_{U_{n-1}}^{U_n}(V^{\otimes r-1})$.

When $q=2$, then $q-1 =1$ so that many of the edges in the Bratteli diagram have weight $1$.  In this case, we can account for the weights of paths in the Bratteli diagram by removing the second condition in the definition of a strict labeling.  A \textit{semi-strict shell tableau} is a shell tableau where we allow $\tau(j,m) = \tau(i,l)+1$ for every pair of arcs $(i,l; \tau(i,l))$ and $(j,m; \tau(j,m))$ in a labeled shell with $\dim(i,l) > \dim(j,m)$ and $i=j$ or $k=l$.  This is reminiscent of semi-standard Young tableaux where we allow the entries along the rows to be weakly increasing.

Suppose $\mathcal{SST}_k(\lambda)$ is the set of semi-strict shell tableaux of length $k$ and shape $\lambda$. Recall the sum of the weights of paths to $\lambda$ is the multiplicity of $\chi^\lambda$ in $V^k$. When $q=2$, this is the number of semi-strict shell tableaux.

 \begin{proposition}
Let $q=2$ and $\lambda \in \hat{Z}_k$.  Then 
\[ \sum_{P \in \mathcal{P}_k(\lambda)} w(P) = |\mathcal{SST}_k(\lambda)| . \]
 \end{proposition}
 
 \begin{proof}
 Let $q=2$ and $\lambda \in \hat{Z}_k$.  Let $T= (\varsigma^1, \ldots, \varsigma^k) \in \mathcal{ST}_k(\lambda)$ be the shell tableau corresponding to the path $P = (\lambda^0, \lambda^{\frac{1}{2}}, \ldots, \lambda^{k-\frac{1}{2}}, \lambda^k)$ via the bijection in the proof of Theorem \ref{tableauxbijection}.  Since the weight of $P$ is the product of its edge labels, it suffices to consider the label of a single edge.  Recall the label of an edge $(\lambda^r,r) \to (\lambda^{r+\frac{1}{2}}, r + \frac{1}{2})$ in $P$ is 
\[m((\lambda^r,r) \to (\lambda^{r+\frac{1}{2}}, r+ \tfrac{1}{2})) = \frac{2^{\crs(\lambda^{r} \cap \lambda^{r+\frac{1}{2}}, \lambda^{r}-\lambda^{r+\frac{1}{2}})}}{2^{\crs(\lambda^{r} \cap \lambda^{r+\frac{1}{2}}, \lambda^{r+\frac{1}{2}}-\lambda^{r})}}.\]

Suppose 
\[ \bigcup_{s=1}^t \{i_s \frown l_s\} \cup \bigcup_{s=1}^{t} \{i_s \smile l_{s+1}\} \ \]
is the shell created by the symmetric difference of $\lambda^{r}$ and $\lambda^{r+\frac{1}{2}}$ where $(i_s, l_{s+1}; a_s) \in \varsigma^{r_s}$.  For $1 \leq s \leq t$, define the set 
\[
Y_s = \left\{m\ \Bigg| \begin{array}{c} (j \frown m, i_s \frown l_s) \in \Crs(\lambda^r \cap \lambda^{r+\frac{1}{2}}, \lambda^r- \lambda ^{r+\frac{1}{2}}),\\
 (j \smile m, i_s \smile l_{s+1}) \notin \Crs(\lambda^r \cap \lambda^{r+\frac{1}{2}}, \lambda^{r+\frac{1}{2}}- \lambda^r)
 \end{array} \right\} .
\]
Note that
\[\sum_{s=1}^t |Y_s| = \crs(\lambda^{r} \cap \lambda^{r+\frac{1}{2}}, \lambda^{r}-\lambda^{r+\frac{1}{2}})-\crs(\lambda^{r} \cap \lambda^{r+\frac{1}{2}}, \lambda^{r+\frac{1}{2}}-\lambda^{r}).  \]

For each subset $X_s \subseteq Y_s$, add the arcs
$(i_s, l; b_{l}) \in \varsigma^{r_s}$ for $l \in X_s $.  There is a unique relabeling of the arcs with a distinct label in $\{1, 2, \ldots, \sum_{r=1}^k|\varsigma^r|+ \sum_{s=1}^t |X_s|\}$ so that the order of the labels of the original arcs in $T$ is preserved, and every pair of arcs $(i, l; \tau(i,l))$ and $(j, m; \tau(j,m))$ in a labeled shell with $\dim(i,l) > \dim(j,m)$ satisfies $\tau(i,l) < \tau(j,m)$.  Then each $(X_1, \ldots, X_t)$ determines one of the $2^{\sum_{s=1}^t |Y_s|}$ semi-strict shell tableaux.

 \end{proof}
 
 \begin{example}
 Consider the path $P$ from Example \ref{extableaux},
 \[
\begin{tikzpicture}
\node at (-1.1,0) {$P = \Bigg(\emptyset, \emptyset,$};
  \draw[fill=black] (0,0) circle (.04);
\draw[fill=black] (.25,0) circle (.04);
\draw[fill=black] (.5,0) circle (.04);
\draw[fill=black] (.75,0) circle (.04);
\draw[fill=black] (1,0) circle (.04);
\draw[fill=black] (1.25,0) circle (.04);
\draw[line width=.15mm] (0,0) to[in=110, out=70, circle, looseness=1.3] (1.25,0);
\node at (1.5,-.1) {,};
  \draw[fill=black] (1.75,0) circle (.04);
\draw[fill=black] (2,0) circle (.04);
\draw[fill=black] (2.25,0) circle (.04);
\draw[fill=black] (2.5,0) circle (.04);
\draw[fill=black] (2.75,0) circle (.04);
\draw[line width=.15mm] (1.75,0) to[in=110, out=70, circle, looseness=1.4] (2.5,0);
\node at (3,-.1) {,};
  \draw[fill=black] (3.25,0) circle (.04);
\draw[fill=black] (3.5,0) circle (.04);
\draw[fill=black] (3.75,0) circle (.04);
\draw[fill=black] (4,0) circle (.04);
\draw[fill=black] (4.25,0) circle (.04);
\draw[fill=black] (4.5,0) circle (.04);
\draw[line width=.15mm] (3.25,0) to[in=110, out=70, circle, looseness=1.4] (4,0);
\draw[line width=.15mm] (3.5,0) to[in=110, out=70, circle, looseness=1.4] (4.5,0);
\node at (4.75,-.1) {,};
  \draw[fill=black] (5,0) circle (.04);
\draw[fill=black] (5.25,0) circle (.04);
\draw[fill=black] (5.5,0) circle (.04);
\draw[fill=black] (5.75,0) circle (.04);
\draw[fill=black] (6,0) circle (.04);
\draw[line width=.15mm] (5,0) to[in=110, out=70, circle, looseness=1.4] (5.75,0);
\draw[line width=.15mm] (5.25,0) to[in=110, out=70, circle, looseness=1.4] (6,0);
\node at (6.25,-.1) {,};
\draw[fill=black] (6.5,0) circle (.04);
\draw[fill=black] (6.75,0) circle (.04);
\draw[fill=black] (7,0) circle (.04);
\draw[fill=black] (7.25,0) circle (.04);
\draw[fill=black] (7.5,0) circle (.04);
\draw[fill=black] (7.75,0) circle (.04);
\draw[line width=.15mm] (6.5,0) to[in=110, out=70, circle, looseness=1.4] (7.25,0);
\draw[line width=.15mm] (6.75,0) to[in=110, out=70, circle, looseness=1.4] (7.75,0);
\draw[line width=.15mm] (7,0) to[in=110, out=70, circle, looseness=1.4] (7.5,0);
\node at (8,-.1) {,};
\draw[fill=black] (8.25,0) circle (.04);
\draw[fill=black] (8.5,0) circle (.04);
\draw[fill=black] (8.75,0) circle (.04);
\draw[fill=black] (9,0) circle (.04);
\draw[fill=black] (9.25,0) circle (.04);
\draw[line width=.15mm] (8.25,0) to[in=110, out=70, circle, looseness=1.4] (9,0);
\draw[line width=.15mm] (8.5,0) to[in=110, out=70, circle, looseness=1.7] (8.75,0);
\draw[line width=.15mm] (8.75,0) to[in=110, out=70, circle, looseness=1.4] (9.25,0);
\node at (9.5,-.1) {,};
\draw[fill=black] (9.75,0) circle (.04);
\draw[fill=black] (10,0) circle (.04);
\draw[fill=black] (10.25,0) circle (.04);
\draw[fill=black] (10.5,0) circle (.04);
\draw[fill=black] (10.75,0) circle (.04);
\draw[fill=black] (11,0) circle (.04);
\draw[line width=.15mm] (9.75,0) to[in=110, out=70, circle, looseness=1.4] (10.5,0);
\draw[line width=.15mm] (10,0) to[in=110, out=70, circle, looseness=1.7] (10.25,0);
\draw[line width=.15mm] (10.25,0) to[in=110, out=70, circle, looseness=1.4] (10.75,0);
\node at (11.25,0) {\Bigg),};
\end{tikzpicture}\]
and corresponding tableaux
\[
\begin{tikzpicture}
\node at (-.75,0) {$T = \Bigg($};
  \draw[fill=black] (0,0) circle (.04);
\draw[fill=black] (.25,0) circle (.04);
\draw[fill=black] (.5,0) circle (.04);
\draw[fill=black] (.75,0) circle (.04);
\draw[fill=black] (1,0) circle (.04);
\draw[fill=black] (1.25,0) circle (.04);
\draw[line width=.15mm] (0,0) to[in=110, out=70, circle, looseness=1.3] (1.25,0);
\draw[line width=.15mm] (0,0) to[in=-110, out=-70, circle, looseness=1.4] (.75,0);
\node at (.625,.6) {\tiny{1}};
\node at (.375,-.45) {\tiny{2}};
\node at (1.5,0) {,};
\draw[fill=black] (1.75,0) circle (.04);
\draw[fill=black] (2,0) circle (.04);
\draw[fill=black] (2.25,0) circle (.04);
\draw[fill=black] (2.5,0) circle (.04);
\draw[fill=black] (2.75,0) circle (.04);
\draw[fill=black] (3,0) circle (.04);
\draw[line width=.15mm] (2,0) to[in=110, out=70, circle, looseness= 1.4] (3,0);
\draw[line width=.15mm] (2,0) to[in=-110, out=-70, circle, looseness= 1.4] (2.75,0);
\draw[line width=.15mm] (2.25,0) to[in=110, out=70, circle, looseness= 1.4] (2.75,0);
\node at (2.55,.55) {\tiny{3}};
\node at (2.375,-.45) {\tiny{4}};
\node at (2.7,.23) {\tiny{6}};
\node at (3.25,-.1) {,};
\draw[fill=black] (3.5,0) circle (.04);
\draw[fill=black] (3.75,0) circle (.04);
\draw[fill=black] (4,0) circle (.04);
\draw[fill=black] (4.25,0) circle (.04);
\draw[fill=black] (4.5,0) circle (.04);
\draw[fill=black] (4.75,0) circle (.04);
\draw[line width= .15mm] (3.75,0) to[in=110, out=70, circle, looseness= 1.4] (4.75,0);
\draw[line width= .15mm] (3.75,0) to[in=-110, out=-70, circle, looseness= 1.7] (4,0);
\node at (4.25, .55) {\tiny{5}};
\node at (3.875, -.25) {\tiny{7}};
\node at (5,-.1) {,};
\draw[fill=black] (5.25,0) circle (.04);
\draw[fill=black] (5.5,0) circle (.04);
\draw[fill=black] (5.75,0) circle (.04);
\draw[fill=black] (6,0) circle (.04);
\draw[fill=black] (6.25,0) circle (.04);
\draw[fill=black] (6.5,0) circle (.04);
\draw[line width=.15mm] (6.5,0) to[in=50, out=130, loop] (6.5,0);
\node at (6.5, .4) {\tiny{8}};
\node at (6.75,0) {$\Bigg).$};
\end{tikzpicture}\]
The path $P$ has weight $2$ since $m((\lambda^3,3) \to (\lambda^{3\frac{1}{2}}, 3\frac{1}{2})) = 2$.  The shell created by the symmetric difference between $\lambda^3$ and $\lambda^{3 \frac{1}{2}}$ is
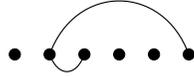
\[\begin{tikzpicture}[scale=1.85]
\draw[fill=black] (6.5,0) circle (.04);
\draw[fill=black] (6.75,0) circle (.04);
\draw[fill=black] (7,0) circle (.04);
\draw[fill=black] (7.25,0) circle (.04);
\draw[fill=black] (7.5,0) circle (.04);
\draw[fill=black] (7.75,0) circle (.04);
\draw[line width=.15mm] (6.75,0) to[in=-110, out=-70, circle, looseness=1.8] (7,0);
\draw[line width=.15mm] (6.75,0) to[in=110, out=70, circle, looseness=1.4] (7.75,0);
\end{tikzpicture}
\]
and the set $Y_1 = \{ 4\}$ as $(1\frown 4, 2 \frown 6) \in \Crs(\lambda^3 \cap \lambda^{3\frac{1}{2}}, \lambda^3 - \lambda^{3\frac{1}{2}})$, but $(1\smile 4, 2 \smile 3) \notin \Crs(\lambda^3 \cap \lambda^{3\frac{1}{2}}, \lambda^{3\frac{1}{2}} - \lambda^{3}).$ The two semi-strict tableaux corresponding to $\emptyset$ and $Y_1$ are
\[
\begin{tikzpicture}
\node at (-.75,0) {$T = \Bigg($};
  \draw[fill=black] (0,0) circle (.04);
\draw[fill=black] (.25,0) circle (.04);
\draw[fill=black] (.5,0) circle (.04);
\draw[fill=black] (.75,0) circle (.04);
\draw[fill=black] (1,0) circle (.04);
\draw[fill=black] (1.25,0) circle (.04);
\draw[line width=.15mm] (0,0) to[in=110, out=70, circle, looseness=1.3] (1.25,0);
\draw[line width=.15mm] (0,0) to[in=-110, out=-70, circle, looseness=1.4] (.75,0);
\node at (.625,.6) {\tiny{1}};
\node at (.375,-.45) {\tiny{2}};
\node at (1.5,0) {,};
\draw[fill=black] (1.75,0) circle (.04);
\draw[fill=black] (2,0) circle (.04);
\draw[fill=black] (2.25,0) circle (.04);
\draw[fill=black] (2.5,0) circle (.04);
\draw[fill=black] (2.75,0) circle (.04);
\draw[fill=black] (3,0) circle (.04);
\draw[line width=.15mm] (2,0) to[in=110, out=70, circle, looseness= 1.4] (3,0);
\draw[line width=.15mm] (2,0) to[in=-110, out=-70, circle, looseness= 1.4] (2.75,0);
\draw[line width=.15mm] (2.25,0) to[in=110, out=70, circle, looseness= 1.4] (2.75,0);
\node at (2.55,.55) {\tiny{3}};
\node at (2.375,-.45) {\tiny{4}};
\node at (2.7,.23) {\tiny{6}};
\node at (3.25,-.1) {,};
\draw[fill=black] (3.5,0) circle (.04);
\draw[fill=black] (3.75,0) circle (.04);
\draw[fill=black] (4,0) circle (.04);
\draw[fill=black] (4.25,0) circle (.04);
\draw[fill=black] (4.5,0) circle (.04);
\draw[fill=black] (4.75,0) circle (.04);
\draw[line width= .15mm] (3.75,0) to[in=110, out=70, circle, looseness= 1.4] (4.75,0);
\draw[line width= .15mm] (3.75,0) to[in=-110, out=-70, circle, looseness= 1.7] (4,0);
\node at (4.25, .55) {\tiny{5}};
\node at (3.875, -.25) {\tiny{7}};
\node at (5,-.1) {,};
\draw[fill=black] (5.25,0) circle (.04);
\draw[fill=black] (5.5,0) circle (.04);
\draw[fill=black] (5.75,0) circle (.04);
\draw[fill=black] (6,0) circle (.04);
\draw[fill=black] (6.25,0) circle (.04);
\draw[fill=black] (6.5,0) circle (.04);
\draw[line width=.15mm] (6.5,0) to[in=50, out=130, loop] (6.5,0);
\node at (6.5, .4) {\tiny{8}};
\node at (6.75,0) {$\Bigg)$};
\end{tikzpicture}\]
and\[
\begin{tikzpicture}
\node at (-.75,0) {$\tilde{T} = \Bigg($};
  \draw[fill=black] (0,0) circle (.04);
\draw[fill=black] (.25,0) circle (.04);
\draw[fill=black] (.5,0) circle (.04);
\draw[fill=black] (.75,0) circle (.04);
\draw[fill=black] (1,0) circle (.04);
\draw[fill=black] (1.25,0) circle (.04);
\draw[line width=.15mm] (0,0) to[in=110, out=70, circle, looseness=1.3] (1.25,0);
\draw[line width=.15mm] (0,0) to[in=-110, out=-70, circle, looseness=1.4] (.75,0);
\node at (.625,.6) {\tiny{1}};
\node at (.375,-.45) {\tiny{2}};
\node at (1.5,0) {,};
\draw[fill=black] (1.75,0) circle (.04);
\draw[fill=black] (2,0) circle (.04);
\draw[fill=black] (2.25,0) circle (.04);
\draw[fill=black] (2.5,0) circle (.04);
\draw[fill=black] (2.75,0) circle (.04);
\draw[fill=black] (3,0) circle (.04);
\draw[line width=.15mm] (2,0) to[in=110, out=70, circle, looseness= 1.4] (3,0);
\draw[line width=.15mm] (2,0) to[in=-110, out=-70, circle, looseness= 1.4] (2.75,0);
\draw[line width=.15mm] (2.25,0) to[in=110, out=70, circle, looseness= 1.4] (2.75,0);
\node at (2.55,.55) {\tiny{3}};
\node at (2.375,-.45) {\tiny{4}};
\node at (2.7,.23) {\tiny{6}};
\node at (3.25,-.1) {,};
\draw[fill=black] (3.5,0) circle (.04);
\draw[fill=black] (3.75,0) circle (.04);
\draw[fill=black] (4,0) circle (.04);
\draw[fill=black] (4.25,0) circle (.04);
\draw[fill=black] (4.5,0) circle (.04);
\draw[fill=black] (4.75,0) circle (.04);
\draw[line width= .15mm] (3.75,0) to[in=110, out=70, circle, looseness= 1.4] (4.75,0);
\draw[line width= .15mm] (3.75,0) to[in=-110, out=-70, circle, looseness= 1.7] (4,0);
\draw[line width= .15mm] (3.75,0) to[in=-110, out=-70, circle, looseness= 1.5] (4.25,0);
\node at (4.25, .55) {\tiny{5}};
\node at (4.05, -.1) {\tiny{8}};
\node at (4.25, -.25) {\tiny{7}};
\node at (5,-.1) {,};
\draw[fill=black] (5.25,0) circle (.04);
\draw[fill=black] (5.5,0) circle (.04);
\draw[fill=black] (5.75,0) circle (.04);
\draw[fill=black] (6,0) circle (.04);
\draw[fill=black] (6.25,0) circle (.04);
\draw[fill=black] (6.5,0) circle (.04);
\draw[line width=.15mm] (6.5,0) to[in=50, out=130, loop] (6.5,0);
\node at (6.5, .4) {\tiny{9}};
\node at (6.75,0) {$\Bigg)$};
\end{tikzpicture}\]
 \end{example}
\noindent respectively.

%
%

\end{document}